\documentclass[11pt]{amsart}
\usepackage{amssymb, latexsym}

\numberwithin{equation}{section}
\newtheorem{Thm}[equation]{Theorem}
\newtheorem{Prop}[equation]{Proposition}

\newtheorem{Lem}[equation]{Lemma}

\theoremstyle{definition}
\newtheorem{Def}[equation]{Definition}
\newtheorem{Exa}[equation]{Example}
\newtheorem{Rmk}[equation]{Remark}

\begin{document}

\title [Weyl Group Multiple Dirichlet Series]
{Weyl Group Multiple Dirichlet Series \\ for Symmetrizable Kac-Moody Root Systems}
\author[Kyu-Hwan Lee]{Kyu-Hwan Lee}
\address{Department of
Mathematics, University of Connecticut, Storrs, CT 06269, U.S.A.}
\email{khlee@math.uconn.edu}
\author[Yichao Zhang]{Yichao Zhang}
\address{Department of
Mathematics, University of Connecticut, Storrs, CT 06269, U.S.A.}
\email{yichao.zhang@uconn.edu}

\begin{abstract}
Weyl group multiple Dirichlet series, introduced by Brubaker, Bump, Chinta, Friedberg and Hoffstein, are expected to be Whittaker coefficients of Eisenstein series on metaplectic groups. Chinta and Gunnells constructed these multiple Dirichlet series for all the finite root systems using the method of averaging a Weyl group action on the field of rational functions.
In this paper, we generalize Chinta and Gunnells' work and construct Weyl group multiple Dirichlet series for the root systems associated with symmetrizable Kac-Moody algebras, and establish their functional equations and meromorphic continuation.
\end{abstract}

\maketitle

\section*{Introduction}

Weyl group multiple Dirichlet series were introduced in the paper \cite{BBCFH} by Brubaker, Bump, Chinta, Friedberg and Hoffstein, and have been studied in their subsequent works \cite{BBF, BBFbk, BBFH}.  These multiple Dirichlet series
are defined for a finite root system $\Phi$ of rank $r$ and a number field $F$ containing the $2n$-th roots of unity, and unify many examples in number theory that have been studied previously on a case-by-case basis, and are expected to be Whittaker coefficients of Eisenstein series on metaplectic groups. This expectation is now called {\em Eisenstein conjecture}, and the conjecture has been proven for the root systems of type $A_r$  \cite{BBFcr}. Moreover, the theory of Weyl group multiple Dirichlet series can be applied to the moments problem of the Riemann zeta function and Dirichlet $L$-functions. A nice survey on this subject can be found in \cite{Bump}.

It is remarkable that there are two distinct constructions of these multiple Dirichlet series. In the work of Brubaker, Bump and Friedberg \cite{BBFbk, BBFcr}, the local coefficients are defined using the data from Kashiwara's crystal graph \cite{Kash}. More precisely, each local coefficient is given as a sum of $G(v)$ over the crystals $v$ of the same weight, where $G(v)$ is a product of Gauss sums that are determined by the crystal graph. The construction has been completed for most of the classical root systems.

The second construction is due to Chinta and Gunnells \cite{CG, CGco}. They defined a new Weyl group action on the field of rational functions in several variables. The action involves Gauss sums and works for all the finite root systems in a coherent way. By taking an average of the Weyl group action, they obtained a deformed Weyl character and defined the local coefficients  using the deformed character.
The two constructions are expected to be equivalent, and the equivalence in type $A$ has been established \cite{BBFcr, CO, McN} .

Since Kac-Moody root systems naturally generalize finite root systems, it was pointed out by Brubaker, Bump and Friedberg \cite{BBFbk, BBFcr} that multiple Dirichlet series can be constructed for any symmetrizable Kac-Moody root systems.
Indeed, in this paper, we generalize  Chinta and Gunnells' construction to the root systems associated with symmetrizable Kac-Moody algebras. So Weyl groups are infinite, in general, and a root can have multiplicity bigger than one. Still, the results of this paper show that the Kac-Moody multiple Dirichlet series have standard properties: absolute convergence, functional equations and meromorphic continuation. 

Nevertheless, the construction in the Kac-Moody case requires new ideas. To generalize the action of the Weyl group, we have to separate the imaginary roots from the real roots and work on a certain sublattice of the root lattice. Although it may be that the field of Laurent series is more natural for the space of the action, it is easily seen that it is not closed under the action. So it is necessary to consider the space of formal distributions. The issue of convergence tells us that the whole space of formal distributions is not closed under the action either. By carefully choosing a subspace of the space of formal distributions, we define the Weyl group action and generalize Chinta and Gunnells's definition in the finite case. This enables us to define the average function of the Weyl group action,  and we use the function to define the local coefficients of our Dirichlet series. Not suprisingly, the issue of absolute convergence arises and becomes crucial at other places. Several lemmas are proved to resolve the absolute convergence at those places.  Finally,  we provide the details in the proof of the meromorphic continuation and functional equations of our multiple Dirichlet series. In particular, we show how to obtain an overlap between regions of analyticity before we apply Bochner's Theorem. It turns out that the continuation is not to all of $\mathfrak h$ but to a convex subcone in the {\em Tits cone} as a consequence of geometric properties of the action of the Weyl group on the Cartan subalgebra. 

It is expected that our multiple Dirichlet series would be related to a Whittaker function up to a normalizing factor in the affine case. However, in the indefinite case, the contribution coming from the imaginary roots can be much more complicated for a Whittaker function, and it is beyond our comprehension at the present. 
 
 Similarly, the multiple Dirichlet series considered in this paper seems to be relevant to the moment problem in the affine case $D_4^{(1)}$. In a recent paper \cite{BD}, Bucur and Diaconu considered the fourth moment of quadratic Dirichlet $L$-functions over rational function fields. They constructed a multiple Dirichlet series, where the group of functional equations is the affine Weyl group $D_4^{(1)}$. They adopted the same averaging method as in this paper, i.e. Chinta and Gunnells's method, and used a deformation of Weyl-Kac denominator function. One can see the similarity between their construction and ours. (See the remark at the end of this paper for more details.) 

This paper consists of five sections. In the first section, we fix notations for Kac-Moody root systems and for Hilbert symbols, power residue symbols and Gauss sums. In Section 2, we show that the Weyl group action on the field of rational functions, defined by Chinta and Gunnells, extends to the case of Kac-Moody root systems and to the set of formal distributions, and we consider a deformed Weyl-Kac character. In Section 3, the local coefficients  of the multiple Dirichlet series are defined using the deformed character, and some estimations for the size of the local coefficients are made. In the next section, we review results and computations in the rank one case and make preparation for a proof of functional equations. In the last section, we collect the results of the previous sections, define Weyl group multiple Dirichlet series from the local coefficients via twisted multiplicativity, and prove functional equations and meromorphic continuation of the multiple Dirichlet series.

\subsection*{Acknowledgments}
The authors are grateful to B. Brubaker, A. Bucur, D. Bump, G. Chinta, A. Diaconu, S. Friedberg,  and P. Gunnells for helpful discussions and suggestions.

\section{Preliminaries}

We refer the reader to \cite{Kac} for a basic theory of Kac-Moody algebras.
Let $A=(a_{ij})_{i,j=1}^r$ be an $r \times r$ symmetrizable generalized Cartan matrix of rank $l$, and $(\mathfrak h, \Delta, \Delta^\vee)$ be a realization of $A$, where $\Delta=\{ \alpha_1, ... , \alpha_r \} \subset \mathfrak h^*$ and  $\Delta^\vee=\{ h_1, ... , h_r \} \subset \mathfrak h$ such that $\alpha_j(h_i)  = a_{ij}$,  $i,j=1, ..., r$. Let $\mathfrak g(A)$ be the symmetrizable Kac-Moody algebra associated to $(\mathfrak h, \Delta, \Delta^\vee)$. Then we have $\dim \mathfrak h =2r -l$. We denote by $\Phi$ the set of roots of $\frak g(A)$ and have $\Phi = \Phi_+ \cup \Phi_-$ where $\Phi_+$ (resp. $\Phi_-$) is the set of positive (resp. negative) roots.  Denote by $\Phi^{\rm{re}}$ (resp. $\Phi^{\rm{im}}$) the set of real (resp. imaginary) roots, and put $\Phi^{\rm{re}}_{\pm} =\Phi^{\rm{re}} \cap \Phi_{\pm}$  and
$\Phi^{\rm{im}}_{\pm} =\Phi^{\rm{im}} \cap \Phi_{\pm}$. We fix a decomposition \begin{equation} \label{eqn-decom-matrix}  A= \mathrm{diag} (\epsilon_1, \cdots , \epsilon_r ) B, \end{equation} where $\epsilon_i$ are positive rational numbers and $B=(b_{ij})$ is a symmetric half-integral matrix, i.e. $b_{ij}=b_{ji}$, $b_{ii} \in \mathbb Z$ and $b_{ij} \in \frac 1 2 \mathbb Z$. We will write $b_i = b_{ii}$.
As in Chapter 2 of \cite{Kac}, we define a standard symmetric bilinear form $(\ , \ )$ on $\mathfrak h^*$  such that $(\alpha_i, \alpha_j)= b_{ij}$ for $i,j =1, ... , r$. 

We extend the sets $\Delta$ and $\Delta^\vee$, and choose  bases \[\Delta \cup \{ \delta_{k} \ | \ k=1, ..., r- l \} \quad \text{ and } \quad \Delta^\vee \cup \{ d_{k} \ | \ k=1, ..., r- l \}\] for $\mathfrak h^*$ and $\mathfrak h$, respectively, such that
$\alpha_j(d_{k}) = 0 \text{ or }1$, $\delta_{k}(h_j)= 0 \text{ or }1$ and $\delta_{k} (d_{k'}) = 0$ for $j=1, ..., r$ and $k, k'=1, ..., r-l$.
Let $P^\vee$ be the $\mathbb Z$-span of the basis $\Delta^\vee \cup \{ d_{k} \ | \ k=1, ..., r- l \}$, and let $\frak h_{\mathbb R} = \mathbb R \otimes P^\vee \subset \frak h$. We set $P= \{ \lambda \in \mathfrak h^* | \lambda(P^\vee) \subseteq \mathbb Z \}$, and $P_+= \{ \lambda \in \mathfrak h^* | \lambda(P^\vee) \subseteq \mathbb Z_{\ge 0} \}$. 
Define $\omega_i \in \frak h^*$ $(i=1, ... , r)$ by $\omega_i(h_j)=\delta_{ij}$, $\omega_i (d_{k})=0$, $j=1, ... , r$, $k=1, ..., r- l$ and put $\rho = \sum_{i=1}^r \omega_i$.
Similarly, we define $\omega^\vee_i \in \frak h$ $(i=1, ... , r)$ by $\alpha_j(\omega^\vee_i)=\delta_{ij}$, 
$\delta_{k}(\omega^\vee_i)=0$, $j=1, ... , r$,  $k=1, ..., r- l$ and put $\rho^\vee = \sum_{i=1}^r \omega^\vee_i $.
Define $Q= \bigoplus_{i=1}^r \mathbb Z \alpha_i$ and $Q_+= \bigoplus_{i=1}^r (\mathbb Z_{\ge 0}) \alpha_i$. 
We have the usual ordering on $P$ (and on $Q$) given by $\lambda \geq \mu \Leftrightarrow \lambda - \mu \in Q_+$.  For $\beta \in Q$, we write $\beta = k_1 \alpha_1 + \cdots + k_r \alpha_r$ and define 
$d(\beta) = \beta(\rho^\vee)=k_1 + \cdots + k_r$.

Let $W$ be the Weyl group of $\mathfrak g(A)$ generated by the simple reflections $\sigma_i \in GL(\mathfrak h^*)$. We have the standard actions of $W$ on $\frak h$ and on $\mathfrak{h}^*$.
For $w \in W$, we let $\Phi(w) = \Phi_+ \cap w^{-1} \Phi_- \subseteq \Phi_+^{\rm{re}}$.  We denote by $\ell(w)$ the length of $w$, and define $\mathrm{sgn}(w) = (-1)^{\ell(w)}$. If $\ell(\sigma_i w) = \ell(w)+1$ then \begin{equation} \label{eqn-phi}  \Phi(\sigma_i w) = \Phi(w) \cup \{ w^{-1} \alpha_i \}, \end{equation} and if $\ell(w \sigma_i) = \ell(w)+1$ then \begin{equation} \label{eqn-phi-1}  \Phi(w \sigma_i) = \sigma_i (\Phi(w)) \cup \{ \alpha_i \}. \end{equation}

\medskip

Let $n \ge 1$ be an integer and let $F$ be an algebraic number field that contains the group $\mu_{2n}$ of
$2n$-th roots of unity. Let $S$ be a finite set of places of $F$ containing the archimedean ones, all those which are
ramified over $\mathbb Q$ and enough others so that the ring $\mathfrak o_S$ of $S$-integers is a principal
ideal domain. We embed $\mathfrak o_S$ into $F_S:=\prod_{v\in S} F_v$ along the diagonal. We choose a nontrivial additive character $\psi$ of $F_S$ such that $\psi(x \mathfrak o_S)=1$ if and only if $x\in \mathfrak o_S$.  Let $S_\infty$ be the set of archimedean places in $S$, and $S_{\mathrm{fin}}$ be the set of the nonarchimedean places so that  $S= S_\infty \cup S_{\mathrm{fin}}$. We denote
$F_\infty =\prod_{v \in S_\infty} F_v$ and $F_{\mathrm{fin}} =\prod_{v \in S_{\mathrm{fin}}} F_v$. Then $F_S=F_\infty \times F_{\mathrm{fin}}$. Let $(x,y)_S = \prod_{v \in S} (x_v,y_v)_v$ be the $S$-Hilbert
symbol, where we take the same convention on $(\ , \ )_v$ as in \cite{BBF}, i.e. it is the inverse of the symbol
used in \cite{Neu}. If $c, d$ are coprime elements of $\mathfrak o_S$, let $(\frac c d )$ denote the $n$-th power residue symbol. Then we have the reciprocity law $\left ( \frac c d \right ) = (d,c)_S \left ( \frac d c \right )$. We fix an isomorphism $\epsilon : \mu_n \rightarrow \{ x \in \mathbb C^\times \ | \ x^n=1 \}$ and will suppress this isomorphism from the notation.
If $t$ is any positive integer and $a, c \in \mathfrak o_S$, we define the Gauss sum \[ g(a,c;t) = \sum_{d \ \mathrm{mod}\  c} \left ( \frac d c \right )^t \psi \left (\frac {ad} c \right ). \]

For $\mathbf{x, y} \in  (F_S^\times)^r$ and for each $i$, we define
\begin{equation} \label{eqn-prev} (\mathbf{x,y})^B_{S,i} = (x_i, y_i)^{b_i}_S \ \prod_{j>i} (x_i, y_j)_S^{2 b_{ij}}, \end{equation} where
$\mathbf{x}=(x_1, \cdots , x_r)$ and $\mathbf{y}=(y_1, \cdots , y_r)$, and set
$(\mathbf{x,y})^B_{S} = \prod_{i=1}^r (\mathbf x,  \mathbf y)^{B}_{S,i}$.
We also define \[\xi_B  ( \mathbf{x}, \mathbf{y}  ) = \prod_{i=1}^r \left ( \frac {x_i} {y_i} \right )^{b_i} \left ( \frac {y_i} {x_i} \right )^{b_i} \ \prod_{i<j}  \left ( \frac {x_i} {y_j} \right)^{2 b_{ij}} \left ( \frac {y_i} {x_j} \right)^{2 b_{ij}},  \]   \[ \left [ \frac {\mathbf{x}} {\mathbf{y}} \right ]^B = \prod_{i=1}^r \left ( \frac {x_i} {y_i} \right )^{b_i}  \qquad \text{ and } \qquad \left [ \frac {\mathbf{x}} {\mathbf{y}} \right ]^{-B} = \prod_{i=1}^r \left ( \frac {x_i} {y_i} \right )^{-b_i} \]  when $\mathbf{x, y} \in (F_S^\times)^r \cap (\frak o_S)^r$.
Let $\Omega = \frak o_S^\times F_S^{\times, n}$ where $F_S^{\times, n}$ is the subgroup of $n$-th powers in $F_S^\times$, and let $\mathcal{M}_B(\Omega)$ be the space of functions $\Psi : (F_S^\times)^r \rightarrow \mathbb C$ such that
$\Psi(\mathbf {e c})= (\mathbf e, \mathbf c)_S^B \ \Psi (\mathbf c)$ when $\mathbf e \in \Omega^r $ and $\mathbf c  \in (F^\times_S)^r$. If $r=1$ and $B=(t)$ we simply write $\mathcal{M}_B(\Omega) = \mathcal{M}_t(\Omega)$.

\section{Weyl Group Action on Formal Distributions}

In this section, we generalize the Weyl group action on the field of rational functions, defined by Chinta and Gunnells, to the case of Kac-Moody root systems and to a subspace of formal distributions. Since the subspace contains all the monomials, we will be able to define $f|_w(\mathbf x)$ for a general formal distribution $f$ and a Weyl group element $w$ whenever the resulting expression is absolutely convergent. In particular, the Weyl group action on the deformed Weyl-Kac character will be well-defined.

\medskip

Let $q$ be a positive integer $\ge 2$. We consider a collection of complex numbers $\gamma(i) \in \mathbb C$, indexed by the integers modulo $n$, and such that $\gamma(0)=-1$ and $\gamma(i) \gamma(-i) = 1/q \qquad \text{if } \ i \not\equiv 0 \mod n$. We also define \[ m(\alpha) = \left\{
\begin{array}{cc}
\frac n {\mathrm{gcd}(n, (\alpha, \alpha) )}  &  \text{ if } \alpha \in \Phi^\text{re}, \\
n  &   \text{ if } \alpha\in \Phi^{\text{im}}. \\
\end{array}\right.
\]

Let $\mathcal{A}= \mathbb C[Q]$ be the group algebra of the lattice $Q$. An element $f \in \mathcal{A}$ can be written as $f=\sum_{\beta \in Q} c(\beta) \, \mathbf x^\beta$ ($c(\beta) \in \mathbb C$) with almost all $c(\beta)$ zero. We identify $\mathcal{A}$ with $\mathbb C[x_1^{\pm 1}, \cdots , x_r^{\pm 1}]$ via $\mathbf x^{\alpha_i} \mapsto x_i$.  We also let $\mathcal{B}= \mathbb C [[x_1, \cdots, x_r ]]$ be the ring of power series. More generally, let $\mathcal{E}= \mathbb C [[x_1^{\pm}, \cdots, x_r^{\pm} ]]$ be the set of formal distributions, which we identify with the set of elements $\sum_{\beta \in Q} c(\beta) \, \mathbf x^\beta$, $c(\beta) \in \mathbb C$,
so that we have $\mathcal A \subset \mathcal E$ and $\mathcal{B} \subset \mathcal {E}$.

Assume $f_i=\sum_{\beta\in Q}c_i(\beta)\, \mathbf x^\beta\in \mathcal{E}$ ($1\leq i\leq k$).  We say the product $f_1\cdots f_k$ is defined if for any $\gamma\in Q$,  the sum
\[C(\gamma)=\sum_{\substack{\beta_1+\cdots+\beta_k=\gamma \\ \beta_i\in Q, 1\leq i\leq k}} c_1(\beta_1)\cdots c_k(\beta_k)\]
is absolutely convergent. If this is the case, we define
\[f_1\cdots f_k=\sum_{\gamma\in Q} C(\gamma)\, \mathbf x^\gamma.\] It follows that if a product is defined in $\mathcal{E}$, it is independent of the order of its factors.
For example, we have \[ \left ( \sum_{k \in \mathbb Z} \mathbf x^{ k \alpha_i} \right ) \left ( \sum_{k \ge 0} \frac {\mathbf x^{ k \alpha_i} }{2^k} \right ) = 2 \sum_{k \in \mathbb Z} \mathbf x^{ k \alpha_i} .\]  However, the following product is not defined:
\[ \left ( \sum_{k \in \mathbb Z} \mathbf x^{ k \alpha_i} \right ) (1-\mathbf x^{\alpha_i}) \left ( \sum_{k \ge 0}  {\mathbf x^{ k \alpha_i} }  \right )  .\]
The use of formal distributions are common in the theory vertex operator algebras. For more details on formal calculus, see, e.g. \cite[Chapter 2]{LeLi}.

We shall say that $f\in\mathcal E$ is {\em invertible} if
$f\in \mathcal B[x_1^{-1},\cdots, x_r^{-1}]\subset\mathcal E$,
and $f$  is invertible in the ring $\mathcal B[x_1^{-1},\cdots, x_r^{-1}]$. Its inverse in this ring will be denoted $f^{-1}$. For example, $1-x_i^{-1}$ is invertible, and its inverse is $\sum_{k=1}^\infty -x_i^k$ (not $\sum_{k=0}^\infty x_i^{-k}$).
Let $\{ f_i (\mathbf x) \}_{i \in I}$ be a family of formal distributions. We define the sum $\sum_{i \in I} f_i(\mathbf x) $ if the coefficient of each $\mathbf x^\beta$ in the sum is an absolutely-convergent series of complex numbers.

Let $\mathcal{E}'$ be the subspace of $\mathcal{E}$ defined by
\[\mathcal{E}'=\{f=\sum_\beta c(\beta)\, \mathbf x^\beta\in\mathcal{E} \ |\ |c(\beta)|\ll q^{d(\beta)}\}.\] Here
$|c(\beta)|\ll q^{d(\beta)}$ means that there exists a positive real number $A$ such that $|c(\beta)|\leq Aq^{d(\beta)}$ for any $\beta\in Q$. Obviously, this subspace contains all the monomials; this is the  subspace where the action of the Weyl group $W$ is always convergent.
For an element $f\in\mathcal E$, denote
\[|f|(\mathbf x)=\sum_{\beta\in Q} |c(\beta)|\ \mathbf x^\beta, \quad \text{if } f(\mathbf x)=\sum_{\beta\in Q} c(\beta)\ \mathbf x^\beta.\]
Trivially, $f\in\mathcal E'$ if and only if $|f|\in\mathcal E'$.
Obviously $\mathcal A\subset \mathcal E'$. We write $\mathcal{B}'=\mathcal{B}\cap\mathcal{E}'$.
Given $f= \sum_{\beta \in Q} c(\beta) \, \mathbf x^\beta \in \mathcal{E}$, we say $\beta\in Q$ is a {\em lower bound} for $f$ if $c(\gamma) \neq 0$ implies $\beta \le \gamma$.
It is easy to see that $f$ has a lower bound if and only if $f\in \mathcal B[x_1^{-1}, \cdots, x_r^{-1}]$.

Let $Q' \subseteq Q$ be the sublattice of $Q$ generated by $m(\alpha)\alpha$ ($\alpha\in \Phi$), namely
\[Q'=\text{span}_\mathbb{Z}\{m(\alpha)\alpha \ | \ \alpha\in\Phi\}=\text{span}_\mathbb{Z}\{m(\alpha)\alpha \ | \ \alpha\in\Phi^\text{re}\},\] by the definition of $m(\alpha)$. It is easy to see that  $Q'$  lies between $nQ$ and $Q$ and is a sublattice of $Q$. It is also not hard to see that it is invariant under the action of $W$ on $Q$, by noting that $m(w \alpha)=m(\alpha)$ and that $\Phi^{\text{re}}$ is invariant under $W$. Let $\nu : Q \rightarrow Q/Q'$ be the projection and define the subspaces $\mathcal{E}'_\beta:=\{f\in \mathcal{E}' \ |\ \nu(\text{supp}(f))\subset\{\beta\}\}  \subset \mathcal E'$ for $\beta\in Q/Q'$. Similarly $\mathcal{B}'_\beta=\mathcal{B}\cap \mathcal{E}'_\beta$, hence $\mathcal{B}'_0$ and $\mathcal{E}'_0$ are defined, and
$\mathcal{E}'=\bigoplus_{\beta\in Q/Q'} \mathcal{E}'_\beta$.
We shall need a lemma on $Q'$.

\begin{Lem} \label{lem-lattice}
(1) For any $1\leq i, j\leq r$, we have \[\frac{2m(\alpha_i)(\alpha_i,\alpha_j)}{m(\alpha_j)(\alpha_j,\alpha_j)}\in\mathbb{Z}.\]

(2) The lattice $Q_1=\text{span}_\mathbb{Z}\{m(\alpha_i)\alpha_i : 1\leq i\leq r\}$ is W-invariant; hence $Q'=Q_1$.

(3) For any $\beta\in Q'$ and $1\leq i\leq r$, $\beta(h_i)\in m(\alpha_i)\mathbb{Z}$ .
\end{Lem}

\begin{proof}
(1) We fix any prime $p$ and denote $v_p$ the valuation at $p$. Assume
$v_p((\alpha_i, \alpha_i))=t_i$, $v_p((\alpha_j, \alpha_j))=t_j$, and $v_p(n)=s$. Since all quantities are integers, we have $t_i, t_j, s\in\mathbb{Z}_{\geq 0}$. Now
\[v_p(m(\alpha_i))=s-\min\{t_i,s\} \quad \text{and}\quad v_p(m(\alpha_j))=s-\min\{t_j,s\} .\]
Since $2(\alpha_i,\alpha_j)$ is divisible by both $(\alpha_i,\alpha_i)$ and $(\alpha_j,\alpha_j)$, we have
$v_p(2(\alpha_i,\alpha_j)):=k\geq \max\{t_i,t_j\}$. What we are trying to show is
$t_j-\min\{t_j,s\}\leq k-\min\{t_i,s\}$, which is obvious by dividing it into two cases: $t_j\leq s$ and $t_j>s$.

(2) It suffices to prove that $\sigma_j(m(\alpha_i)\alpha_i)\in Q_1$ for any $i, j$.  It holds trivially if $i=j$. Now assume $i\neq j$. By explicit calculations,
\[\sigma_j(m(\alpha_i)\alpha_i)=m(\alpha_i)\alpha_i-\frac{2m(\alpha_i)(\alpha_i,\alpha_j)}{m(\alpha_j)(\alpha_j,\alpha_j)} m(\alpha_j)\alpha_j,\]
which is in $Q_1$ by Part (1). Hence $Q_1$ is W-invariant, and by definition of  real roots, $Q'=Q_1$.

(3) It is enough to show this for generators of $Q'$. Since $Q_1=Q'$, we need only to prove it for the generators of $Q_1$, in which case this is nothing but Part (1).
\end{proof}

We write $\mathbf{x}=(x_1, ... , x_r)$, and define a change-of-variable formula by
$(\sigma_i \mathbf{x})_j = (qx_i)^{- a_{ij}} x_j$ for a simple reflection $\sigma_i \in W$, where $A=(a_{ij})$ is the generalized Cartan matrix. One can check that this definition extends to the whole group $W$. For $\beta = \sum k_i \alpha_i \in Q$ and $w \mathbf x = (y_1, ... , y_r)$, $w \in W$, we define $(w \mathbf x)^\beta = y_1^{k_1} \cdots y_r^{k_r}$.  Then  we have \begin{equation} \label{eqn-inv} (w \mathbf{x})^\beta =q^{d(w^{-1}\beta -\beta)} \mathbf{x}^{w^{-1} \beta} \quad \text{ for } w \in W .  \end{equation} In particular, we obtain $(\sigma_i \mathbf x)^\beta = (q x_i)^{- \beta(h_i)} \mathbf x^{\beta}$.

In the rest of this section, we fix $\lambda \in P_+$.
We define a shifted action of $W$ on $Q$ (treating $Q$ as a set) by \[ \sigma_i \cdot \beta = \sigma_i (\beta -\lambda -\rho)+\lambda +\rho, \qquad \beta \in Q, \quad i=1,..., r. \] For any $\beta \in Q$, we set \[ \mu_{i,\lambda}(\beta) = \mu_i(\beta)=(\lambda +\rho -\beta)(h_i), \qquad i=1, ..., r.\] Then we have \begin{eqnarray} &&\label{eqn-see} \mu_i(\beta) = \mu_i(0) - \beta(h_i), \quad  \sigma_i \cdot \beta = \beta + \mu_i(\beta) \alpha_i = \sigma_i \beta +\mu_i (0) \alpha_i, \quad \mu_i(\sigma_i \cdot \beta) = - \mu_i(\beta), \\ &&
 \quad \sigma_i\cdot (\beta+\gamma)=\sigma_i\cdot \beta+\sigma_i\gamma, \quad \text{ and } \quad \mu_i(\beta+\gamma)=\mu_i(\beta)-\gamma(h_i) .\nonumber\end{eqnarray}

Now we define the action of $W$ on $\mathcal{E}'$. First, fix a generator $\sigma_i \in W$ and put $m=m(\alpha_i)$ for the moment. For an integer $k$, we denote by $[k]_{m}$ the largest multiple of $m$ that is smaller than or equal to $k$. We define, for any $\beta \in Q$, \[ \mathcal{P}_{\beta, i} (\mathbf x) = (q \mathbf x^{\alpha_i})^{[\mu_i(\beta)]_m}  (1 - 1/q)\sum_{k=0}^\infty (q^{m-1}\mathbf x^{m\alpha_i})^k, \quad \text{ and}\] \[ \mathcal{Q}_{\beta, i}(\mathbf x) =   \gamma(b_i \mu_i(\beta) ) q^{\mu_i(\beta)} (1- (q \mathbf x^{\alpha_i})^{-m})\sum_{k=0}^\infty (q^{m-1}\mathbf x^{m\alpha_i})^k .\] Note that $ \mathcal{P}_{\beta, i} (\mathbf x)$ and $\mathcal{Q}_{\beta, i}(\mathbf x)$ belong to $\mathcal E'_0$.

\begin{Rmk}
The definitions of $\mathcal{P}_{\beta, i}$ and $\mathcal{Q}_{\beta, i}$ are slightly different from those in \cite{CGco}, but the definition of the action of $\sigma_i$ will be the same as in  \cite{CGco}.
\end{Rmk}

\begin{Def} \label{def-action}
For  $\beta \in Q$ and each $i=1, ... , r$, we define
\begin{equation*}  (\mathbf{x}^\beta |_\lambda \sigma_i)(\mathbf x) =  \mathcal{P}_{\beta, i}(\mathbf x) \mathbf{x}^\beta  + \mathcal{Q}_{\beta, i}(\mathbf x)   \mathbf{x}^{\sigma_i \cdot \beta} . \end{equation*} Moreover, we extend it to $\mathcal E'$ by
\[\left.\left(\sum_{\beta\in Q}c(\beta)\, \mathbf x^\beta\right)\right|_\lambda\sigma_i(\mathbf x)=\sum_{\beta\in Q}c(\beta)\, (\mathbf x^\beta|_\lambda\sigma_i)(\mathbf x).\] Finally, if $f\in \mathcal E'$ and $w=\sigma_{i_1}\cdots\sigma_{i_k}\in W$, we define
\[f|_\lambda w (\mathbf x)=f|_\lambda\sigma_{i_1}|_\lambda\sigma_{i_2}\cdots |_\lambda\sigma_{i_k}(\mathbf x).\]
\end{Def}

We will show that the above definition indeed defines an action of $W$ on $\mathcal E'$. In particular, $\left(\mathbf x^\beta |_\lambda w\right)(\mathbf x)$ will be well-defined, and the following definition has no ambiguity:

\begin{Def} \label{def-dis}
If $f(\mathbf x)=\sum_{\beta\in Q} c(\beta)\mathbf x^\beta \in \mathcal E$, then for any $w\in W$, we define
\[f|_\lambda w(\mathbf x)=\sum_{\beta\in Q} c(\beta)\left(\mathbf x^\beta |_\lambda w\right)(\mathbf x),\]
provided the sum on the right-hand side is absolutely convergent. 
\end{Def}

Now we begin to prove that the action of $W$ on $\mathcal E'$ is well-defined.
It is easy to see that for any $\beta$ and $i$, $\mathbf x^\beta |_\lambda\sigma_i(\mathbf x)\in \mathcal E'$. However, this is not enough for our purpose and we need to prove more:

\begin{Lem}\label{lem-act-1}
If $f\in\mathcal E'$, so is $f|_\lambda\sigma_i$ for any $i=1, \dots, r$. 
\end{Lem}

\begin{proof}
We need to show that in the expression of $f|_\lambda\sigma_i(\mathbf x)$, all coefficients converge absolutely, and the resulting distribution is in $\mathcal E'$.

Put $m=m(\alpha_i)$ and assume $f(\mathbf x)=\sum_{\beta\in Q}c(\beta)\, \mathbf x^\beta$. By definition,
\begin{eqnarray*}
f|_\lambda\sigma_i(\mathbf x) &=& \sum_{\beta} c(\beta) \mathbf x^\beta (1-q^{-1}) (q\mathbf x^{\alpha_i})^{[\mu_i(\beta)]_m}\sum_{k=0}^\infty(q^{m-1}\mathbf x^{m\alpha_i})^k\\
&+& \sum_{\beta} c(\beta) \mathbf x^{\sigma_i\cdot\beta} \gamma(b_i\mu_i(\beta)) q^{\mu_i(\beta)}\sum_{k=0}^\infty(q^{m-1}\mathbf x^{m\alpha_i})^k\\
&+& (-1)\sum_{\beta} c(\beta) \mathbf x^{\sigma_i\cdot\beta} \gamma(b_i\mu_i(\beta)) q^{\mu_i(\beta)}(q\mathbf x^{\alpha_i})^{-m}\sum_{k=0}^\infty(q^{m-1}\mathbf x^{m\alpha_i})^k.\end{eqnarray*}

For any $\gamma\in Q$, denote
\begin{eqnarray*}
I_1(\gamma) &=& \{(k,\beta) \ |\  \beta+([\mu_i(\beta)]_m+mk)\alpha_i=\gamma, k\in\mathbb{Z}_{\geq 0}, \beta \in Q\},\\
I_2(\gamma) &=& \{(k,\beta) \ |\  \sigma_i\cdot\beta+mk\alpha_i=\gamma, k\in\mathbb{Z}_{\geq 0}, \beta \in Q\},\\
I_3(\gamma) &=& \{(k,\beta) \ |\  \sigma_i\cdot\beta+m(k-1)\alpha_i=\gamma, k\in\mathbb{Z}_{\geq 0}, \beta \in Q\}.
\end{eqnarray*}
It is easy to see that in each set, $k$ and $\beta$ determine each other. Then the coefficient of $\mathbf x^\gamma$ in $f|_\lambda\sigma_i(\mathbf x)$ can be written as
$c_1(\gamma)+c_2(\gamma)+c_3(\gamma)$, where
\begin{eqnarray*}
c_1(\gamma)&=&\sum_{(k,\beta)\in I_1(\gamma)} c(\beta)(1-q^{-1})q^{[\mu_i(\beta)]_m}q^{(m-1)k},\\
c_2(\gamma)&=&\sum_{(k,\beta)\in I_2(\gamma)} c(\beta)\gamma(b_i\mu_i(\beta)) q^{\mu_i(\beta)}q^{(m-1)k},\\
c_3(\gamma)&=&\sum_{(k,\beta)\in I_3(\gamma)} (-1)c(\beta)\gamma(b_i\mu_i(\beta)) q^{-m}q^{\mu_i(\beta)}q^{(m-1)k}.
\end{eqnarray*}
We denote $c_j'(\gamma)$, for $j=1,2,3$, the one obtained from $c_j(\gamma)$ by replacing all terms in the sum with their absolute values.

We prove now $c_j'(\gamma)<\infty$ and $c_j'(\gamma)\ll q^{d(\gamma)}$ simultaneously. Here the former gives $f|_\lambda\sigma(\mathbf x)\in \mathcal E$ and the latter implies it also belongs to $\mathcal E'$, since $|c_j(\gamma)|\leq c_j'(\gamma)$.

If $j=1$ and $(k,\beta)\in I_1(\gamma)$, we have
$\beta+([\mu_i(\beta)]_m+mk)\alpha_i=\gamma$. Since we have $|c(\beta)|\ll q^{d(\beta)}=q^{d(\gamma)-[\mu_i(\beta)]_m-mk}$, we obtain
\[c_1'(\gamma)\ll \sum_{k=0}^\infty q^{d(\gamma)-k}\ll q^{d(\gamma)}.\] Similarly, we can deal with the case $j=2$ and $j=3$. 
\end{proof}

\begin{Rmk} 
The proof of Lemma \ref{lem-act-1} actually shows that  the $k$-multiple sum for the coefficient of $\mathbf x^\gamma$ in $f|_\lambda\sigma_{i_1}|_\lambda\sigma_{i_2}\cdots |_\lambda\sigma_{i_k}(\mathbf x)$ is absolutely convergent  for each $\gamma$. To see this,  we first replace $f$ with $|f|$, and at each step of applying one simple reflection, we replace all $c_j(\gamma)$'s by $c_j'(\gamma)$'s as in the proof. At each step we have an element in $\mathcal E'$ with all coefficients positive, and in the end we have an element in $\mathcal E'$, say $g$.
On the other hand, the sum for any coefficient in any of the resulting products should be absolutely bounded by the corresponding coefficient in  $g$, hence absolutely convergent. 
\end{Rmk}

\begin{Lem}\label{lem-act-3}
(1) Let $w\in W$. Then $f(\mathbf x)\in \mathcal E'$ if and only if $f(w\mathbf x)\in \mathcal E'$; and $f(\mathbf x)\in \mathcal E'_0$ if and only if $f(w\mathbf x)\in \mathcal E'_0$.

(2) Let $f_1\in \mathcal E'_0$, $f_2\in \mathcal E'$ and assume $|f_1|\, |f_2|\in \mathcal E'$. Then for any $i$, we have
\begin{equation} \label{eqn-wq}
(f_1f_2)|_\lambda\sigma_i(\mathbf x)=f_1(\sigma_i\mathbf x) (f_2|_\lambda\sigma_i)(\mathbf x).
\end{equation} In particular, this holds if $f_1$ or $f_2$ is a monomial.
\end{Lem}

\begin{proof}
For (1), we only need to prove one direction from each statement. By the definition of change of variable, if $f(\mathbf x)=\sum_\beta c(\beta)\mathbf x^\beta$,
\[f(w\mathbf x):=\sum_\beta c_w(\beta)\mathbf x^\beta=\sum_\beta c(w\beta)q^{d(\beta-w\beta)}\mathbf x^\beta.\]
Then the coefficients of $f(w\mathbf x)$
\[|c_w(\beta)|=|c(w\beta)|q^{d(\beta-w\beta)}\ll q^{d(w\beta)}q^{d(\beta-w\beta)}=q^{d(\beta)}.\] Hence the first statement is true. The second statement is now obvious, since $Q'$ is invariant under the action of $W$.

For (2), we first prove the case when $f_1$ is a monomial. Let $m=m(\alpha_i)$. The assumption $|f_1|\, |f_2|\in \mathcal E'$ is valid and both sides of \eqref{eqn-wq} are defined. By the definition of the action of $\sigma_i$, it suffices to assume that $f_2$ is also a monomial. Let $f_1(\mathbf x)=\mathbf x^\beta$, $\beta\in Q'$ and $f_2(\mathbf x)=\mathbf x^\alpha$. Then by Lemma \ref{lem-lattice}, (\ref{eqn-see}) and the fact that $n|b_im$, 
it is easy to see that \[ \begin{array}{lcl} \sigma_i \cdot (\beta + \alpha) = \sigma_i \beta+ \sigma_i \cdot \alpha, & \quad& \mu_i(\beta+\alpha) = \mu_i(\alpha) - \beta(h_i), \\ \gamma(b_i \mu_i(\beta+\alpha)) = \gamma(b_i \mu_i(\alpha)), &\quad& [ \mu_i(\beta+\alpha)]_m = [\mu_i(\alpha)]_m - \beta(h_i).
\end{array}
\] Since we have  $(\sigma_i \mathbf x ) ^{\beta} = (q \mathbf x^{\alpha_i})^{-\beta(h_i)} \mathbf x^{\beta} = q^{-\beta(h_i)} \mathbf x^{\sigma_i \beta}$, we obtain
\begin{eqnarray*}
 (\mathbf x^{\beta+\alpha} |_\lambda \sigma_i)(\mathbf x) &=& \mathcal{P}_{\beta+\alpha, i}(\mathbf x) \mathbf{x}^{\beta+\alpha}  + \mathcal{Q}_{\beta+\alpha, i}(\mathbf x)   \mathbf{x}^{\sigma_i \cdot (\beta+\alpha)} \\ &=&
(q \mathbf x^{\alpha_i})^{[\mu_i(\alpha)]_m - \beta(h_i)}  (1 - 1/q)\sum_{k=0}^\infty(q^{m-1}\mathbf x^{m\alpha_i})^k\mathbf x^{\beta + \alpha} \\ & & + \gamma(b_i \mu_i(\alpha) ) q^{\mu_i(\alpha)-\beta(h_i)}  (1- (q \mathbf x^{\alpha_i})^{-m})\sum_{k=0}^\infty(q^{m-1}\mathbf x^{m\alpha_i})^k \mathbf x^{\sigma_i \beta + \sigma_i \cdot \alpha} \\ &=& (\sigma_i \mathbf x)^{\beta} \mathcal{P}_{\alpha, i}(\mathbf x) \mathbf{x}^\alpha  + (\sigma_i \mathbf x)^{\beta}  \mathcal{Q}_{\alpha, i}(\mathbf x)   \mathbf{x}^{\sigma_i \cdot \alpha} \\ &=& (\sigma_i \mathbf x)^{\beta }  (\mathbf x^{\alpha} |_\lambda \sigma_i)(\mathbf x) .
\end{eqnarray*}

Let us consider the general case. Assume
\[f_1(\mathbf x)=\sum_{\beta\in Q'} c_1(\beta) \mathbf x^\beta, \quad f_2(\mathbf x)=\sum_{\beta\in Q} c_2(\beta) \mathbf x^\beta, \quad \text{and}\quad (f_1f_2)(\mathbf x)=\sum_{\beta\in Q} c(\beta) \mathbf x^\beta.\] Formally,
\begin{eqnarray*}
(f_1f_2)|_\lambda\sigma_i (\mathbf x) &=& \sum_{\beta\in Q} c(\beta) (\mathbf x^\beta|_\lambda\sigma_i)(\mathbf x)\\
&=& \sum_{\beta_1,\beta_2 \in Q}c_1(\beta_1)c_2(\beta_2)(\mathbf x^{\beta_1}\mathbf x^{\beta_2})|_\lambda\sigma_i(\mathbf x)\\
&=& \sum_{\beta_1,\beta_2 \in Q}c_1(\beta_1)c_2(\beta_2)\mathbf (\sigma_i \mathbf x)^{\beta_1}\mathbf x^{\beta_2}|_\lambda\sigma_i(\mathbf x)\\
&=& f_1(\sigma_i\mathbf x) f_2|_\lambda\sigma_i(\mathbf x),
\end{eqnarray*}
by the monomial case we just proved. What is missing in the calculation above is the well-definedness of the product in the last line. However, by the assumption that $|f_1|\, |f_2|\in\mathcal E'$ and elementary calculations, we can see that all coefficients in this product are absolutely bounded by the corresponding coefficients in $(|f_1|\, |f_2|)|_\lambda\sigma_i(\mathbf x)$ which belongs to $\mathcal E'$ by  Lemma \ref{lem-act-1}. Hence the lemma follows.
\end{proof}

\begin{Prop} \label{prop-action}
Let $W$ be the Weyl group of the symmetrizable Kac-Moody root system $\Phi$ with symmetrization (\ref{eqn-decom-matrix}).
The action of $\sigma_i$'s on $\mathcal E'$ is compatible with the defining relation of $W$; that is, the action of $W$ is well-defined on $\mathcal E'$.
\end{Prop}

\begin{proof}
The group $W$ is a Coxeter group. More precisely, $W$ is generated by $\sigma_i$ ($i=1, ..., r$) and the defining relations are $\sigma_i^2=1$, $(\sigma_i \sigma_j )^{m_{ij}}=1$ for $1\leq i\neq j\leq r$, where  $m_{ij} = 2  \text{ if } a_{ij}a_{ji} =0$;   $m_{ij} = 3  \text{ if } a_{ij}a_{ji} =1$;  $m_{ij} = 4  \text{ if } a_{ij}a_{ji} =2$;  $m_{ij} = 6 \text{ if } a_{ij}a_{ji} =3$;  $m_{ij} = \infty  \text{ if } a_{ij}a_{ji} \ge 4$.   Here $x^\infty=1$ for any $x$ by notational convention. Therefore, we need only to consider the following four cases:  \[ \left \{ \begin{array}{lll} a_{ij} =0, & a_{ji}=0, & b_i \text{ and } b_j \text{ are arbitrary;} \\ a_{ij} = -1, & a_{ji}= -1,  & b_i =b_j ;\\ a_{ij} = -1, & a_{ji}= -2 , & b_i =2b_j; \\ a_{ij} = -1, & a_{ji}= -3 , & b_i =3 b_j . \end{array} \right . \]

We need to verify the relation $\mathbf x^\beta |_\lambda \sigma_i^2 = \mathbf x^\beta$ for each $i$, and then we need to check in each of the four cases whether the action of generators $\sigma_i$ and $\sigma_j$ is compatible with the defining relation. It is obvious that we have a very similar situation as in the finite case (Theorem 3.2 \cite{CGco}). Roughly speaking, we have to verify the same identities on $\mathcal P$'s and $\mathcal Q$'s as Chinta and Gunnells did therein (up to a slight modification caused from our modification on the $\mathcal P$'s and $\mathcal Q$'s). We will make this precise and reduce our verification to theirs.

Recall the notation $\mathcal A=\mathbb{C}[Q]\subset \mathcal E'$. Consider the multiplicative set $S$ in $\mathcal A$ generated by the set $\{1-q^{m_i-1}(w\mathbf x)^{m_i\alpha_i}: w\in W, 1\leq i\leq r\}$, and let $S^{-1}\mathcal A$ be the corresponding localized ring, namely
\[S^{-1}\mathcal A=\mathcal A[(1-q^{m_i-1}(w\mathbf x)^{m_i\alpha_i})^{-1}: w\in W, 1\leq i\leq r].\]
On the other hand, we consider the following ring inside $\mathcal E'$:
\[\mathcal R=\mathcal A\left[\sum_{k=0}^\infty (q^{m_i-1}(w\mathbf x)^{m_i\alpha_i})^k: w\in W, 1\leq i\leq r\right].\]
Let us first verify that $\mathcal R$ is indeed a ring, that is, all finite products in $\mathcal R$ are defined and belong to $\mathcal E'$. The generators are all of the form
\[f_{\beta,i}(\mathbf x)=\sum_{k=0}^\infty q^{-k}q^{km_id(\beta)}\mathbf x^{km_i\beta}.\] Then a general finite product (i.e., a monomial in $\mathcal R$) is of the form
\[g(\mathbf x)f_{\beta_1,i_1}(\mathbf x)\cdots f_{\beta_l,i_l}(\mathbf x), \quad g(\mathbf x)=\sum_{\beta}c(\beta)\mathbf x^\beta\in \mathcal A.\]
The summation in the coefficient $C(\gamma)$ of $\mathbf x^\gamma$ in this product is absolutely bounded by
\begin{eqnarray*}
&&\sum_{\substack{\beta, k_1,\cdots, k_l \ge 0\\ \beta+k_1m_{i_1}\beta_1+\cdots+k_l m_{i_l}\beta_l=\gamma}} |c(\beta)|q^{-k_1-\cdots-k_l}q^{d(\gamma)-d(\beta)}\\
&=& \sum_{\beta}\sum_{\substack{k_1,\cdots,k_l \ge 0\\ k_1m_{i_1}\beta_1+\cdots+k_l m_{i_l}\beta_l=\gamma-\beta}} |c(\beta)| q^{-k_1-\cdots-k_l}q^{d(\gamma)-d(\beta)}\\
&\ll & \sum_{\beta}\sum_{k_1,\cdots,k_l \ge 0}|c(\beta)| q^{-k_1-\cdots-k_l}q^{d(\gamma)-d(\beta)}\\
&\ll & \sum_{\beta} |c(\beta)|q^{-d(\beta)} q^{d(\gamma)}\\
&\ll & q^{d(\gamma)}, 
\end{eqnarray*}
since $g \in \mathcal A$ and $c(\beta)=0$ for almost all $\beta$.
Therefore $\mathcal R$ is a ring. 

Consider now the inclusion $\mathcal A\hookrightarrow \mathcal R$. Since all elements in $S$ are invertible in $\mathcal R$ under such inclusion, there exists a unique injective ring homomorphism $S^{-1}\mathcal A\rightarrow \mathcal R$, which is compatible with the inclusion. (Actually, this is an isomorphism.) It follows that any identity in $S^{-1}\mathcal A$ has a counterpart in $\mathcal R$. In particular, all those identities of $\mathcal P$'s and $\mathcal Q$'s in $S^{-1}\mathcal A$, verified by Chinta and Gunnells, can be carried over here in $\mathcal R$. Note that their $\mathcal P$'s
and $\mathcal{Q}$'s in $S^{-1}\mathcal A$ have the images that are essentially equal to our $\mathcal P$'s and $\mathcal Q$'s in $\mathcal R$. (Notice that our minor modification on $\mathcal P$'s and $\mathcal Q$'s, which is just a redistribution of some monomials to make them belong to $\mathcal E'_0$, does not give rise to any problem.) 

Consequently, the compatibility of the $W$-action on monomials follows from that of the finite case and we refer the reader to Chinta and Gunnells' computations in Theorem 3.2 of \cite{CGco}.
\end{proof}

\begin{Rmk}
With the above theorem established, we recall Definition \ref{def-dis} and note that 
the formal distribution 
$(f|_\lambda w)(\mathbf x)$
is well-defined for any $w\in W$ and $f\in \mathcal E$
if it is absolutely convergent.  However, as mentioned in the introduction, this definition does not give an action  on the whole space $\mathcal E$.
\end{Rmk}

We denote the multiplicity of $\alpha \in \Phi$ by $\mathrm{mult}(\alpha)$ and define
\begin{eqnarray} \label{DDe}& &  \Delta (\mathbf x) = \prod_{\alpha \in \Phi_+} ( 1 - q^{m(\alpha) d(\alpha)} \mathbf x^{m(\alpha) \alpha})^{\mathrm{mult}(\alpha)},   \\ \text{and} & & 
D (\mathbf x) = \prod_{\alpha \in \Phi_+} ( 1 - q^{m(\alpha) d(\alpha)-1} \mathbf x^{m(\alpha) \alpha})^{\mathrm{mult}(\alpha)} .  \nonumber \end{eqnarray} Note that $\Delta(\mathbf x) , D( \mathbf x) \in \mathcal{B}_0$.

\begin{Lem}[Compare with Lemma 3.3 in \cite{CGco}] \label{lem-delta} \hfill
\begin{enumerate} \item The formal distribution $\Delta(w \mathbf x)$ is invertible in $\mathcal{E}$ for each $w \in W$, and we have
\[  \frac {\Delta(\mathbf x)}{ \Delta(w \mathbf x) } =  \mathrm{sgn} (w) q ^{ d(\beta)} \mathbf x^{\beta} ,\] where $\beta = \sum_{\alpha \in \Phi (w)} m(\alpha) \alpha$. 

\item Let $ j(w, \mathbf x) =  \Delta(\mathbf x) / \Delta(w \mathbf x)$. Then the function $j(w, \mathbf x)$ satisfies the cocycle relation
\[ j(ww', \mathbf x) = j(w, w' \mathbf x) j(w' , \mathbf x) \qquad \text{ for }w, w' \in W.\]
\end{enumerate}
\end{Lem}

\begin{proof} (1)
Since $m(w \alpha) =m(\alpha)$ and $\mathrm{mult}(w \alpha) = \mathrm{mult}(\alpha)=\mathrm{mult}(-\alpha) $, we have
\begin{eqnarray*}
\Delta(w \mathbf x) &=& \prod_{\alpha \in \Phi_+} ( 1 - q^{m(\alpha) d(\alpha)} (w\mathbf x)^{m(\alpha) \alpha})^{\mathrm{mult}(\alpha)} \\ &=& \prod_{\alpha \in \Phi_+} ( 1 -  q^{ m(\alpha) d( w^{-1} \alpha )  } \mathbf x^{ m(\alpha)( w^{-1} \alpha ) } )^{\mathrm{mult}(\alpha)} \qquad (\text{using }(\ref{eqn-inv}))  \\ & =& \prod_{\alpha \in \Phi_+ \atop w^{-1} \alpha \in \Phi_-} ( 1 -  q^{ m(\alpha) d(w^{-1} \alpha)  } \mathbf x^{ m(\alpha)(w^{-1} \alpha ) } )^{\mathrm{mult}(\alpha) }  \\ & & \hskip 3cm \times \prod_{\alpha \in \Phi_+ \atop w^{-1} \alpha \in \Phi_+} ( 1 -  q^{ m(\alpha) d(w^{-1} \alpha)  } \mathbf x^{ m(\alpha)(w^{-1} \alpha ) } )^{\mathrm{mult}(\alpha)}  \\ & =& \prod_{\alpha \in \Phi_-  \atop  w \alpha \in \Phi_+ } ( 1 -  q^{ m(\alpha) d(\alpha)  } \mathbf x^{ m(\alpha) \alpha } )^{\mathrm{mult}(\alpha)} \prod_{  \alpha \in \Phi_+   \atop w \alpha \in \Phi_+ } ( 1 -  q^{ m(\alpha) d(\alpha) } \mathbf x^{ m(\alpha) \alpha } )^{\mathrm{mult}(\alpha)}
\\ & =& \prod_{\alpha \in \Phi(w) } ( 1 -  q^{ - m(\alpha) d(\alpha)  } \mathbf x^{-  m(\alpha) \alpha } )^{\mathrm{mult}(\alpha)} \prod_{  \alpha \in \Phi_+   \atop w\alpha \in \Phi_+ } ( 1 -  q^{ m(\alpha) d(\alpha) } \mathbf x^{ m(\alpha) \alpha } )^{\mathrm{mult}(\alpha)}.
\end{eqnarray*}
It follows from $\Phi(w) \subset \Phi^{\mathrm{re}}$ that $\mathrm{mult}(\alpha)=1$ for each $\alpha \in \Phi(w)$, and we obtain
\[ \Delta(w \mathbf x) = \left (  \prod_{\alpha \in \Phi(w)} - \, q^{ - m(\alpha) d(\alpha)  } \mathbf x^{- m(\alpha) \alpha } \right ) \Delta (\mathbf x).\]
Since $|\Phi(w)|= \ell(w) < \infty$, we see that $\Delta(w \mathbf x)$ is invertible in
${\mathcal E}$. Now we have
\[ \frac {\Delta(\mathbf x)}{\Delta ( w \mathbf x)} = \prod_{\alpha \in \Phi(w)} - \, q^{ m(\alpha) d(\alpha)  } \mathbf x^{m(\alpha) \alpha } = \mathrm{sgn} (w) q ^{ d(\beta)} \mathbf x^{\beta} ,\] where $\beta = \sum_{\alpha \in \Phi (w)} m(\alpha) \alpha$.

(2) It is straightforward to verify the identity, using the definition of $j$. It can also be proved directly using the part (1).
\end{proof}

We need the following lemma to show that some elements to be defined later belong to $\mathcal B$.

\begin{Lem} \label{lem-cal}
Let $\lambda \in P_+$ and $\beta \in Q$.
\begin{enumerate}
\item The function $(w \mathbf x)^{-\beta} j(w, \mathbf x) (\mathbf x^\beta |_\lambda w) (\mathbf x)$ is an element of $\mathcal{B}$ for $w \in W$.

\item For any $w\in W$, the lattice point
$\beta=\sum_{\gamma\in \Phi(w)} \gamma$
is a lower bound for $j(w,\mathbf x)(1|_\lambda w)(\mathbf x)$.
\end{enumerate}
\end{Lem}

\begin{proof} We prove (1) and (2) simultaneously using induction.
If $w=1$, there is nothing to prove for both of (1) and (2).  

Assume that $\ell(\sigma_i w) = \ell(w)+1$. Then, using Lemma \ref{lem-delta} (2), we obtain
\begin{eqnarray*}
& & (\sigma_i w \mathbf x)^{-\beta} j(\sigma_i  w, \mathbf x) (\mathbf x^\beta |_\lambda \sigma_i w) (\mathbf x) \\ &=& (\sigma_i w \mathbf x)^{-\beta} j(\sigma_i  , w \mathbf x) j(w, \mathbf x) \left [  \left ( \mathcal{P}_{\beta, i}(\mathbf x) \mathbf{x}^\beta  + \mathcal{Q}_{\beta, i}(\mathbf x)   \mathbf{x}^{\sigma_i \cdot \beta} \right ) |_\lambda w \right ] \\  &=& - (\sigma_i w \mathbf x)^{-\beta} q^{m(\alpha_i)} ( w \mathbf x)^{m(\alpha_i) \alpha_i} j(w, \mathbf x) \left [ \mathcal{P}_{\beta, i}(w \mathbf x) ( \mathbf{x}^\beta |_\lambda w )  + \mathcal{Q}_{\beta, i}(w \mathbf x)  (  \mathbf{x}^{\sigma_i \cdot \beta}  |_\lambda w )  \right ] .
\end{eqnarray*}

We first consider the term having $\mathcal P$ factor. By induction, we have
$(w \mathbf x)^{-\beta} j(w, \mathbf x) (\mathbf x^\beta |_\lambda w) (\mathbf x) \in\mathcal{B}$ and we need only to consider \begin{eqnarray*}  & & (\sigma_i w \mathbf x)^{-\beta} ( w \mathbf x)^{m(\alpha_i) \alpha_i} ( w \mathbf x)^{[\mu_i(\beta)]_{m(\alpha_i)} \alpha_i} ( w \mathbf x)^{\beta} \\ &=& q^{\beta(h_i)}  (w \mathbf x)^{- \sigma_i \beta + m(\alpha_i) \alpha_i + [\mu_i(\beta)]_{m(\alpha_i)} \alpha_i +\beta } \\ &=& q^{\beta(h_i)}  (w \mathbf x)^{\beta (h_i) \alpha_i + m(\alpha_i) \alpha_i + [\mu_i(0) -\beta(h_i) ]_{m(\alpha_i)} \alpha_i } .
 \end{eqnarray*}
We see that \begin{eqnarray*}  & &  \beta (h_i)+ m(\alpha_i)  + [\mu_i(0) -\beta(h_i) ]_{m(\alpha_i)} \\ &=& m(\alpha_i) +  [\mu_i(0) -\beta(h_i) ]_{m(\alpha_i)} -(\mu_i(0) -\beta(h_i)) + \mu_i(0)   > \mu_i(0) .\end{eqnarray*}
Since $w^{-1} \alpha_i >0$ by (\ref{eqn-phi}) and $\mu_i(0)\geq 1$, we have proved that the term having $\mathcal P$ factor is an element of $\mathcal B$. Putting $\beta =0$, we also prove, by induction, that the term having $\mathcal P$ factor in this case has a lower bound 
\[w^{-1}\alpha_i+\sum_{\beta\in \Phi(w)}\beta=\sum_{\beta\in\Phi(\sigma_i w)}\beta.\]

Now we consider the term having $\mathcal Q$ factor. Again by induction, we need only to consider \[
(\sigma_i w \mathbf x)^{-\beta} ( w \mathbf x)^{m(\alpha_i) \alpha_i} ( w \mathbf x)^{- m(\alpha_i) \alpha_i} ( w \mathbf x)^{\sigma_i \cdot \beta} = q^{\beta(h_i)} ( w \mathbf x)^{- \sigma_i \beta+ \sigma_i \cdot \beta} = q^{\beta(h_i)} ( w \mathbf x)^{\mu_i(0) \alpha_i},  \] and see that the term is an element of $\mathcal B$ as well. Putting $\beta=0$ again, we have the same lower bound as above. This completes the induction.
\end{proof}

Since $\lambda \in P_+$, we have $\mu_i(0)>0$ and consider the sum
\[ s(\mathbf x,\lambda)=\sum_{w \in W} j(w, \mathbf x) (1 |_\lambda w) (\mathbf x). \] 

\begin{Lem}
$s(\mathbf x,\lambda)$ is an element of $\mathcal{B}$.
\end{Lem} 
\begin{proof}
Lemma \ref{lem-cal} says that we only need to show that $s(\mathbf x,\lambda)$ is well-defined. Consider any $\gamma\in Q$.  It is clear from \ref{lem-cal} that if $\ell(w)$ is big enough, then our lower bound obtained there will be large enough to exclude the $\mathbf x^\gamma$ term in $j(w, \mathbf x) (1 |_\lambda w) (\mathbf x)$. In other words, only finitely many terms in $\mathbf s(\mathbf x,\lambda)$ contribute to the coefficient of $\mathbf x^\gamma$, hence the series is absolutely convergent. 
\end{proof}

Note that $\Delta (\mathbf x)$ is a unit in $\mathcal B$. We define \[ h( \mathbf x ; \lambda) =  \Delta(\mathbf x)^{-1}s(\mathbf x,\lambda)=\Delta(\mathbf x)^{-1} \sum_{w \in W} j(w, \mathbf x) (1 |_\lambda w) (\mathbf x) \in  \mathcal B \quad \text{ and }\] \[ N(\mathbf x; \lambda) = h(\mathbf x ; \lambda) D( \mathbf x) \in \mathcal B .\]

\begin{Rmk}
The function $h(\mathbf x ; \lambda)$ should be considered as a deformed Weyl-Kac character. More precisely, when $n=1$, the change of variables $q \mathbf x^{\alpha_i} \mapsto \mathbf z^{-\alpha_i}$ for each $i$ makes the function $\mathbf z^{\lambda} h(\mathbf x ; \lambda)$ the Weyl-Kac character of the irreducible representation $V(\lambda)$ of the Kac-Moody algebra $\mathfrak g(A)$. See \cite{Kac} for the details on Weyl-Kac characters.
\end{Rmk}

\begin{Prop}[See Theorem 3.5 in \cite{CGco}] \label{prop-inv} The distribution $(h |_\lambda w)(\mathbf x,\lambda)$ is well-defined and  $(h |_\lambda w)(\mathbf x,\lambda) = h(\mathbf x ; \lambda)$, for each $w \in W$. 
\end{Prop}

\begin{proof}
Fix any $w'\in W$, and we first prove that $(s|_\lambda w')(\mathbf x,\lambda)$ is well-defined, and $(s|_\lambda w')(\mathbf x,\lambda)=j(w',\mathbf x)^{-1}s(\mathbf x,\lambda)$. Indeed, if we write $\left(1|_\lambda w\right)(\mathbf x)=\sum_{\beta}c_w(\beta)\mathbf x^\beta$, then
\begin{eqnarray*}
(s|_\lambda w')(\mathbf x,\lambda)&=& \sum_{\beta}\sum_{w}c_w(\beta) j(w,w'\mathbf x) (\mathbf x^\beta|_\lambda w')(\mathbf x)\\
&=& \sum_{w}\sum_{\beta}c_w(\beta) j(w,w'\mathbf x) (\mathbf x^\beta|_\lambda w')(\mathbf x)\\
&=& j(w',\mathbf x)^{-1}\sum_{w} j(ww',\mathbf x) (1|_\lambda w w')(\mathbf x)\\
&=& j(w',\mathbf x)^{-1}s(\mathbf x, \lambda),
\end{eqnarray*}
where the switch of the two summations can be justified by absolute convergence as follows. 

From Lemma \ref{lem-cal}, we know that $j(w,\mathbf x)(1|_\lambda w)(\mathbf x)$ has a lower bound given by the sum of roots in $\Phi(w)$; namely, all the nonzero terms should have exponents bigger than such a sum of roots. Write such a bound as
$\beta(w)$, and we see that $d(\beta(w))\rightarrow\infty, \quad\text{as } \ell(w)\rightarrow\infty$.

Now we can prove the absolute convergence by showing that for each $\gamma\in Q$ there are at most finitely many $w\in W$ such that
\[ \sum_{\beta}c_w(\beta) j(w,w'\mathbf x) (\mathbf x^\beta|_\lambda w')(\mathbf x) = \left (j(w,\mathbf x)(1|_\lambda w)(\mathbf x)\right)|_\lambda w'\]
contributes to the coefficient of $\mathbf x^\gamma$. We know that $j(w,\mathbf x)$ is monomial supported on $Q'$, so
\[ \sum_{\beta}c_w(\beta) j(w,w'\mathbf x) (\mathbf x^\beta|_\lambda w')(\mathbf x) = j(w,w'\mathbf x)(1|_\lambda ww')(\mathbf x)=j(w',\mathbf x)^{-1}\left(j(ww',\mathbf x)(1|_\lambda ww')(\mathbf x)\right).\]
This distribution is bounded below by $\beta(ww')-\beta(w')$, and with $w'$ fixed we have
\[d(\beta(ww')-\beta(w'))\rightarrow\infty, \quad\text{as } \ell(w)\rightarrow\infty.\]
So if $\ell(w)$ is large, the contribution of
\[\left(j(w,\mathbf x)(1|_\lambda w)(\mathbf x)\right)|_\lambda w'\]
 to the  coefficient of $\mathbf x^\gamma$ is zero. Therefore the switching of the double sum is justified.

Let us prove the proposition. Assume that
\[\Delta(\mathbf x)^{-1}=\sum_{\beta\in Q'}a(\beta)\mathbf x^\beta \text{ \  and \ } s(\mathbf x, \lambda)=\sum_{\beta\in Q} c(\beta)\mathbf x^\beta.\] Then we have
\begin{eqnarray*}
(h |_\lambda w)(\mathbf x,\lambda) &=& \sum_{\gamma\in Q}\left(\sum_{\beta\in Q'}a(\beta)c(\gamma-\beta)\right) (\mathbf x^\gamma|_\lambda w)(\mathbf x)\\
&=& \sum_{\gamma\in Q}\left(\sum_{\beta\in Q'}a(\beta)c(\gamma-\beta)\right) (w\mathbf x)^\beta (\mathbf x^{\gamma-\beta}|_\lambda w)(\mathbf x)\\
&=& \Delta(w\mathbf x)^{-1} (s|_\lambda w)(\mathbf x,\lambda)= h(\mathbf x,\lambda),
\end{eqnarray*} and the convergence also follows from this. \end{proof}

Fix an simple root $\alpha_i$ and let $m=m(\alpha_i)$ for the time being. We write $N(\mathbf x; \lambda) = \sum_{\mu \in Q} a_{\mu} \mathbf x^{\mu}$. Given any $\beta \in Q$ , we set \[ S_{\beta, i} = \{ \beta + k m \alpha_i \ | \ k \in \mathbb Z \}, \] and define \[ N_{\beta, i} (\mathbf x)  = \sum_{\mu \in S_{\beta ,i }} a_\mu \mathbf x^\mu \in \mathcal{B} .\] Now choose $\beta \in Q$ and define \[ f_{\beta, i}(\mathbf x) = \begin{cases} \displaystyle{\frac {N_{\beta, i}(\mathbf x) - \gamma ( - b_i \mu_i(\beta)) (q \mathbf x^{\alpha_i})^{(- \mu_i (\beta))_m} N_{\sigma_i \cdot \beta, i}(\mathbf x) } { 1 - q^{m-1} \mathbf x^{m \alpha_i} } } & \text{if } m \nmid \mu_i(\beta); \\ \displaystyle{\frac { N_{\beta, i}(\mathbf x) }{ 1 -q^{m-1} \mathbf x^{m \alpha_i} } } & \text{otherwise}.
\end{cases}\] Here, as before, we denote by $(k)_m$ the remainder upon division of $k$ by $m$. By our convention, $1-q^{m-1}x_i^m$ is invertible in $\mathcal E$, with inverse $\sum_{k=0}^\infty q^{k(m-1)}x_i^{km}$.

\begin{Prop}[Compare with Theorem 3.6 in \cite{CGco}] \label{prop-frac}
We have
\[  {f_{\beta, i}( \sigma_i \mathbf x) } = \begin{cases} (q \mathbf x^{\alpha_i} )^{(\mu_i(\beta))_m - \mu_i(0) }  {f_{\beta, i}( \mathbf x) }& \text{if } m \nmid \mu_i(\beta); \\(q \mathbf x^{\alpha_i})^{m- \mu_i(0)} {f_{\beta, i}(  \mathbf x) }& \text{otherwise}.\end{cases} \]
\end{Prop}

\begin{proof}
Assume that $m \nmid \mu_i(\beta)$. We write $N_{\beta, i} (\mathbf x) = \left ( \sum_{k \in \mathbb Z } a_{\beta + k m \alpha_i} \mathbf x^{k m \alpha_i} \right ) \mathbf x^\beta$ and
define \[ B_{\beta, i}(\mathbf x) =\sum_{k \in \mathbb Z } a_{\beta + k m \alpha_i} \mathbf x^{k m \alpha_i}  \in \mathcal{B}_0,\] so that we have $N_{\beta, i}(\mathbf x) = B_{\beta, i} (\mathbf x) \mathbf x^\beta$.
we also define \[ F_{\beta, i} (\mathbf x) = \frac{N_{\beta, i}(\mathbf x) + N_{\sigma_i \cdot \beta, i}(\mathbf x)}{1 - q^{m-1} \mathbf x^{m \alpha_i}} = \frac{B_{\beta, i}(\mathbf x) \mathbf x^\beta + B_{\sigma_i \cdot \beta, i}(\mathbf x) \mathbf x^{\sigma_i \cdot \beta} }{1 - q^{m-1} \mathbf x^{m \alpha_i}}.\] We obtain from Proposition \ref{prop-inv} that \[ \frac {N(\mathbf x ; \lambda) } {1 - q^{m-1} \mathbf x^{m \alpha_i} } \] is invariant under the action of $|_\lambda \sigma_i$. Indeed, the action is well-defined by the same argument as in the proof of Proposition \ref{prop-inv}. The invariance follows from that of $h(\mathbf x,\lambda)$ and the invariance of $D(\mathbf x)(1-q^{m-1}\mathbf x^{m\alpha_i})^{-1}$ under change of variable $\mathbf x\rightarrow \sigma_i\mathbf x$.
Then it implies that $F_{\beta , i}$ is invariant under $|_\lambda \sigma_i$. On the other hand, applying $\sigma_i$ to $F_{\beta, i}$, we get \[ F_{\beta, i}(\mathbf x) = (F_{\beta, i}|_\lambda \sigma_i)(\mathbf x) = \frac{B_{\beta, i}(\sigma_i \mathbf x ) (\mathbf x^\beta |_\lambda \sigma_i) (\mathbf x) + B_{\sigma_i \cdot \beta, i}(\sigma_i \mathbf x) \mathbf ( x^{\sigma_i \cdot \beta}|_\lambda \sigma_i)(\mathbf x) }{1 - q^{-m-1} \mathbf x^{-m \alpha_i}}. \] Using this, we compute further and obtain \[ f_{\beta, i}(\mathbf x) = \frac{B_{\beta, i}(\sigma_i \mathbf x ) (q \mathbf x^{\alpha_i} )^{[\mu_i(\beta)]_m} \mathbf x^\beta- B_{\sigma_i \cdot \beta, i}(\sigma_i \mathbf x) \gamma(-b_i \mu_i(\beta)) (q \mathbf x^{\alpha_i})^{-m - \mu_i(\beta)}
\mathbf x^{\sigma_i \cdot \beta} }{1 - q^{-m-1} \mathbf x^{-m \alpha_i}}. \] Now the assertion of the proposition follows from this.

The proof for the case that $m | \mu_i(\beta)$ is similar, and we omit the details.
\end{proof}

\section{The Coefficients $H$}

In this section we define the coefficients $H$ and find bounds for them. The coefficients $H$ will be the essential data in defining Weyl group multiple Dirichlet series.

\medskip

We specialize $\gamma(i)$ to be \[ \gamma(i) = \begin{cases} g(1, \varpi; i)/q & \text{if } i \not\equiv 0 \,(\mathrm{mod}\ n); \\ -1 & \text{otherwise} , \end{cases} \] where $\varpi$ is a prime in
$\mathfrak o_S$ and $q$ is the norm of $\varpi$ in $\mathfrak o_S$. We define \[ \varpi^\beta_Q = (\varpi^{k_1}, \cdots , \varpi^{k_r}) \in (\mathfrak o_S)^r \quad \text{ and } \quad \varpi^\lambda_P = (\varpi^{l_1}, \cdots , \varpi^{l_r}) \in (\mathfrak o_S)^r \] where $\beta = \sum_{i=1}^r k_i \alpha_i \in Q_+$ and $\lambda = \sum_{i=1}^r l_i \omega_i \in P_+$.  Denote the $\mathbf x^\beta$-coefficient of $N(\mathbf x; \lambda)$ by \begin{equation} \label{eqn-h1} H(\varpi_Q^\beta; \varpi_P^\lambda).\end{equation}

We fix a generator $\varpi$ for each prime ideal of $\mathfrak o_S$.
For $\mathbf a = (a_1, \cdots , a_r) \in (\mathfrak o_S)^r$, we have decompositions \begin{equation} \label{eqn-decom} \mathbf a = \mathbf u \prod_{\varpi} \varpi^{\beta_\varpi}_Q = \mathbf u \prod_{\varpi } \varpi^{\lambda_\varpi}_P , \qquad \mathbf u \in \left ( \mathfrak o_S^\times \right )^r,  \end{equation} with $\beta_\varpi \in Q_+$ and $\lambda_\varpi \in P_+$ for each prime $\varpi$. We define $H(\mathbf c; \mathbf m)$ for any $\mathbf c, \mathbf m \in  (\mathfrak o_S)^r  \cap (F_S^\times)^r$ in what follows. First we set $H(\mathbf u; \mathbf m)=1$ for $\mathbf u \in \left ( \mathfrak o_S^\times \right )^r$.
 If we have $\mathrm{gcd}(c_1 \cdots c_r, c'_1 \cdots c'_r) =1$ for $\mathbf c=(c_1, \cdots, c_r)$  and $\mathbf c'=(c'_1, \cdots, c'_r)$, we require that the twisted multiplicativity should hold:
\begin{equation} \label{eqn-h2} H(\mathbf{cc'}; \mathbf m) = \xi_B ( \mathbf c , \mathbf c'  ) H(\mathbf{c}; \mathbf m) H(\mathbf{c'}; \mathbf m).\end{equation} We also require the relation \begin{equation} \label{eqn-h3} H(\mathbf c; \mathbf{mm'}) = \left [ \frac{\mathbf m'}{\mathbf c} \right ]^{-B} H(\mathbf c ; \mathbf m) \end{equation} if $\mathrm{gcd}(c_1 \cdots c_r, m'_1 \cdots m'_r) =1$ for $\mathbf c=(c_1, \cdots, c_r)$ and $\mathbf m'=(m'_1, \cdots, m'_r)$. Now note that we have defined the coefficients $H(\mathbf c; \mathbf m)$ for any $\mathbf c, \mathbf m \in (\mathfrak o_S)^r  \cap (F_S^\times)^r$.

We fix $q$ for the time being. Let $\frak h_{\mathbb R} = \mathbb R \otimes P^\vee \subset \mathfrak h$. For $\mathbf s \in \frak h$, we write $\mathbf s = \frak{Re}(\mathbf s)+ \sqrt{-1} \ \frak{Im}(\mathbf s)$ with $\frak{Re}(\mathbf s), \frak{Im} (\mathbf s) \in \frak h_{\mathbb R}$.
We define the evaluation map $EV_q : \mathcal{E} \times \mathfrak h \rightarrow \mathbb C$ by
\[ EV_q \left ( \sum_\beta c(\beta) \mathbf x^\beta, \mathbf s \right ) = \sum_\beta c(\beta) q^{-\beta(\mathbf s)}, \] whenever it is convergent. Similarly, we define $|EV|_q : \mathcal{E} \times \mathfrak h \rightarrow \mathbb C$ by
\[ |EV|_q \left ( \sum_\beta c(\beta) \mathbf x^\beta, \mathbf s \right ) = \sum_\beta |c(\beta) q^{-\beta(\mathbf s)}| = \sum_\beta |c(\beta) | q^{-\beta(\frak{Re}(\mathbf s))} , \] whenever it is convergent.

\begin{Prop} \label{prop-abc}
Let $\beta \in Q$ and $w \in W$, and suppose that
\begin{equation} \label{eqn-important}
 \mathfrak{Re}(\alpha_i(\mathbf s)) >1 \qquad \text{ for each } i =1, ..., r.
\end{equation}
Then we have
\[ |EV|_q \left ( j(w, \mathbf x) (\mathbf x^\beta |_\lambda w) ,  \mathbf s \right )  \le 3^{\ell(w)} q^{-d(\beta)}q^{(w^{-1} \cdot \beta )(\rho^\vee - \mathfrak {Re} (\mathbf s))} .\]
\end{Prop}

\begin{proof} We may assume that $\mathbf s$ is real, i.e. $\mathbf s =\mathfrak{Re}(\mathbf s)$.
We use induction on the length of $w$. If $w=1$ then
\[ \left | EV \right |_q \left ( j(1, \mathbf x) (\mathbf x^\beta |_\lambda 1) , \mathbf s \right ) = \left | EV \right |_q (\mathbf x^\beta) = q^{-\beta(\mathbf s)} . \]
Now assume that $\ell(\sigma_j w) = \ell(w)+1$ for $w \in W$ and write $m=m(\alpha_j)$. We see from (\ref{eqn-phi}) that $w^{-1} \alpha_j$ is a positive root, and we get $w^{-1}\alpha_j(\rho^\vee - \mathbf s) <0$ from the assumption.
We consider
\begin{eqnarray}
\left | EV \right |_q \left ( \mathcal P_{\beta, j}(w \mathbf x) , \mathbf s \right ) &=& \left | EV  \right |_q \left ( \left ( q (w \mathbf x)^{\alpha_j} \right )^{[\mu_j(\beta)]_m} \frac{1 - 1/q}{1- \left ( q (w \mathbf x)^{\alpha_j} \right )^m /q } , \ \mathbf s \right ) \nonumber \\ &=& \left ( q^{ d(w^{-1} \alpha_j)} q^{-w^{-1}(\alpha_j)(\mathbf s)} \right )^{ [\mu_j (\beta)]_m} \frac {1 - 1/q}{1- \left ( q^{d(w^{-1} \alpha_j)} q^{-w^{-1}(\alpha_j)(\mathbf s) } \right )^m /q } \nonumber  \\ &=& \left ( q^{ w^{-1} \alpha_j(\rho^\vee - \mathbf s) } \right )^{ [\mu_j (\beta)]_m} \frac {1 - 1/q}{1- \left ( q^{w^{-1} \alpha_j(\rho^\vee - \mathbf s) } \right )^m /q } \nonumber \\ & \le & q^{[\mu_j (\beta)]_m  w^{-1} \alpha_j(\rho^\vee - \mathbf s) } , \label{eqn-well}
\end{eqnarray}
and
\begin{eqnarray}
\left | EV \right |_q \left ( \mathcal Q_{\beta, j}(w \mathbf x) , \mathbf s \right ) &\le& \left | EV \right |_q \left (q^{\mu_j(\beta)} \frac{1 - \left ( q (w \mathbf x)^{\alpha_j} \right )^{-m}}{1- \left ( q (w \mathbf x)^{\alpha_j} \right )^m /q } , \ \mathbf s \right ) \nonumber   \\ &=&  q^{\mu_j(\beta)} \left | EV \right |_q \left (   - \left ( q (w \mathbf x)^{\alpha_j} \right )^{-m} +   \frac{1 - 1/q }{1- \left ( q (w \mathbf x)^{\alpha_j} \right )^m /q } , \ \mathbf s \right ) \nonumber     \\   & \le & q^{\mu_j(\beta)}\left (  q^{ -m \, w^{-1} \alpha_j(\rho^\vee - \mathbf s)}   + \frac{1-1/q}{1-  q^{m \,  w^{-1} \alpha_j(\rho^\vee - \mathbf s)}/q } \right )  \nonumber  \\ & = & q^{\mu_j (\beta) - m \, w^{-1} \alpha_j(\rho^\vee - \mathbf s) } \  \frac{1+ q^{ m \, w^{-1} \alpha_j(\rho^\vee - \mathbf s)} (1-2/q) }{1-  q^{ m \,  w^{-1} \alpha_j(\rho^\vee - \mathbf s)}/q }  \nonumber \\ & \le & 2 q^{\mu_j (\beta) - m \, w^{-1} \alpha_j(\rho^\vee - \mathbf s) } . \label{eqn-well-1}
\end{eqnarray}
Combining  (\ref{eqn-well}) and (\ref{eqn-well-1}) with the computation \begin{eqnarray*} j(\sigma_j w , \mathbf x ) (\mathbf x^\beta |_\lambda \sigma_j w) &=& j(\sigma_j , w \mathbf x) j(w , \mathbf x) \left [  \left ( \mathcal  P_{\beta , j} (\mathbf x) \mathbf x^\beta + \mathcal Q_{\beta, j} \mathbf x^{\sigma_j \cdot \beta} \right ) |_\lambda w \right ] \\ &=& - q^m (w \mathbf x)^{m \alpha_j} j(w, \mathbf x) \left [ \mathcal P_{\beta, j}(w \mathbf x) (\mathbf x^\beta |_\lambda w) + \mathcal Q_{\beta, j}(w \mathbf x) ( \mathbf x^{\sigma_j \cdot \beta}|_\lambda w) \right ],
\end{eqnarray*}
we obtain by induction
\begin{eqnarray*}
 & & \left | EV \right |_q \left ( j(\sigma_j w, \mathbf x) (\mathbf x^\beta |_\lambda \sigma_j w) , \mathbf s \right ) \\ & \le & q^{(m+ [\mu_j (\beta)]_m) w^{-1} \alpha_j(\rho^\vee - \mathbf s)}  \ \left | EV \right |_q \left (  j(w, \mathbf x)(\mathbf x^\beta |_\lambda w)  , \mathbf s \right )  \\ & & \phantom{LLLLLLLLLLL} + 2 q^{\mu_j(\beta)}  \ \left  | EV \right |_q \left (  j(w, \mathbf x)(\mathbf x^{\sigma_j \cdot \beta} |_\lambda w)  , \mathbf s \right )   \\ & \le & 3^{\ell(w)} q^{-d(\beta)}q^{(m+ [\mu_j (\beta)]_m)  w^{-1} \alpha_j(\rho^\vee - \mathbf s) +(w^{-1} \cdot \beta )(\rho^\vee - \mathbf s) } \\ & & \phantom{LLLLLLLLLLL}  + 2\cdot 3^{\ell(w)} q^{-d(\beta)}q^{(w^{-1} \cdot \sigma_j \cdot \beta ) (\rho^\vee - \mathbf s)}  .
\end{eqnarray*}
The exponent of $q$ in the first term becomes
\begin{eqnarray*} & & (m+ [\mu_j (\beta)]_m)  w^{-1} \alpha_j(\rho^\vee - \mathbf s) +(w^{-1} \cdot \beta )(\rho^\vee - \mathbf s) \\ & \leq & \mu_j(\beta)  w^{-1} \alpha_j(\rho^\vee - \mathbf s)+(w^{-1} \cdot \beta )(\rho^\vee - \mathbf s)\\ & =&\left ( (\sigma_j w)^{-1} \cdot \beta ) \right ) (\rho^\vee -\mathbf s)  .\end{eqnarray*}
The exponent of $q$ in the second term is
\[ (w^{-1} \cdot \sigma_j \cdot \beta ) (\rho^\vee - \mathbf s) = \left ( (\sigma_j w)^{-1} \cdot \beta ) \right ) (\rho^\vee -\mathbf s) .\]

Therefore,  we obtain
\[
\left | EV \right |_q \left ( j(\sigma_j w, \mathbf x) (\mathbf x^\beta |_\lambda \sigma_j w) , \mathbf s \right ) \le 3^{\ell(w)+1} q^{-d(\beta)}q^{\left ( (\sigma_j w)^{-1} \cdot \beta ) \right )(\rho^\vee - \mathbf s)}.
\]
\end{proof}

Given $\beta \in Q_+$, we define \[ N^{(\varpi)}_{\beta, i} (\mathbf x; \mathbf m) = \sum_{j \ge 0} H(\varpi_Q^{\beta + jm\alpha_i}; \mathbf m) \mathbf x^{\beta + jm\alpha_i} ,\] where $m= m(\alpha_i)$. Let $\lambda \in P_+$ be such that $\varpi^\lambda_P$ is the $\varpi$-factor in the decomposition of $\mathbf m$, i.e. $\lambda = \lambda_\varpi$ in  (\ref{eqn-decom}) . We write \[ \mathbf m' = \mathbf m / \varpi_P^\lambda = (m'_1, \cdots , m'_r) ,\] and set $\mu_i(\beta)=\mu_{i,\lambda}(\beta) = (\lambda + \rho -\beta)(h_i)$ as before. We put $q= |\varpi |$
and  define \[ f^{(\varpi)}_{\beta, i} (\mathbf x; \mathbf m) = \begin{cases} \displaystyle{\frac{N^{(\varpi)}_{\beta, i} (\mathbf x; \mathbf m) - q^{-1} g(m'_i, \varpi; -b_i \mu_i(\beta)) (q \mathbf x^{\alpha_i})^{(-\mu_i(\beta))_m} N^{(\varpi)}_{\sigma_i \cdot \beta, i} (\mathbf x; \mathbf m) } {1 - q^{m-1} \mathbf x^{m \alpha_i} } } & \text{if } \mu_i(\beta) \nmid m ; \\ \displaystyle{\frac{N^{(\varpi)}_{\beta, i} (\mathbf x; \mathbf m)  } {1 - q^{m-1} \mathbf x^{m \alpha_i} } } & \text{otherwise }.
\end{cases} \]

\begin{Thm} [Compare with Theorem 4.1 in \cite{CGco}] \label{thm-3-last}
We have
\[{f^{(\varpi)}_{\beta, i}( \sigma_i \mathbf x ; \mathbf m ) }= \begin{cases} (q \mathbf x^{\alpha_i} )^{ (\mu_i(\beta))_m - \mu_i(0)}  {f^{(\varpi)}_{\beta, i} (\mathbf x; \mathbf m ) } & \text{if } m \nmid \mu_i(\beta); \\(q \mathbf x^{\alpha_i})^{m-\mu_i(0) } {f^{(\varpi)}_{\beta, i} (\mathbf x; \mathbf m ) } & \text{otherwise}.\end{cases} \]
\end{Thm}

\begin{proof}

We first establish the following identities:
\begin{equation} \label{eqn-jjj}
 N^{(\varpi)}_{\beta, i} (\mathbf x; \mathbf m) = \left [ \frac{\mathbf m'}{\varpi^\beta_Q} \right ]^{-B} N_{\beta, i}( \mathbf x) \quad \text{ and } \quad f^{(\varpi)}_{\beta, i} (\mathbf x; \mathbf m) = \left [ \frac{\mathbf m'}{\varpi^\beta_Q} \right ]^{-B} f_{\beta, i}( \mathbf x) .
\end{equation}

From the twisted multiplicativity, we obtain  \begin{eqnarray*} H(\varpi_Q^{\beta +j m \alpha_i}; \mathbf m) &=& H(\varpi_Q^{\beta +j m \alpha_i};
 \varpi_P^\lambda \mathbf m')  \\ &=&   \left [ \frac{\mathbf m'}{\varpi^{\beta+jm\alpha_i}_Q} \right ]^{-B} H(\varpi_Q^{\beta +j m \alpha_i}; \varpi_P^\lambda ) \\ & =&   \left [ \frac{\mathbf m'}{\varpi^{\beta}_Q} \right ]^{-B} H(\varpi_Q^{\beta +j m \alpha_i}; \varpi_P^\lambda ) ,
\end{eqnarray*} since $\displaystyle{ \left ( \frac{ m'_i} {\varpi^{jm}} \right )^{-b_i} =\left ( \frac{ m'_i} {\varpi^{j}} \right )^{-b_i m } = 1 }$.
Then we have \[ N^{(\varpi)}_{\beta, i} (\mathbf x; \mathbf m) = \left [ \frac{\mathbf m'}{\varpi^\beta_Q} \right ]^{-B}
\sum_{j \ge 0} H(\varpi_Q^{\beta +j m \alpha_i}; \varpi_P^\lambda ) \mathbf x^{\beta + jm \alpha_i} = \left [ \frac{\mathbf m'}{\varpi^\beta_Q} \right ]^{-B} N_{\beta, i} (\mathbf x) .\]

Since we have $\sigma_i \cdot \beta = \beta +\mu_i(\beta) \alpha_i$ and  the multiplicativity of the power residue symbol, we obtain \[   N^{(\varpi)}_{\sigma_i \cdot \beta, i} (\mathbf x; \mathbf m) = \left [ \frac{\mathbf m'}{\varpi^{\sigma_i \cdot \beta}_Q} \right ]^{-B} N_{\sigma_i \cdot \beta, i} (\mathbf x) = \left [ \frac{\mathbf m'}{\varpi^{\beta}_Q} \right ]^{-B} \left ( \frac{m'_i}{\varpi } \right )^{- b_i \mu_i(\beta ) } N_{\sigma_i \cdot \beta, i} (\mathbf x) .\] On the other hand, \[ g(m'_i , \varpi ; - b_i \mu_i(\beta)) =  \left ( \frac{m'_i}{\varpi } \right )^{ b_i \mu_i(\beta) } g(1, \varpi; -b_i \mu_i(\beta)).\] Now it it straightforward to see that
\[f^{(\varpi)}_{\beta, i} (\mathbf x; \mathbf m) = \left [ \frac{\mathbf m'}{\varpi^\beta_Q} \right ]^{-B} f_{\beta, i}( \mathbf x) .\]

Thus we have proved the identities in (\ref{eqn-jjj}).
Now the theorem follows from (\ref{eqn-jjj}) and Proposition \ref{prop-frac}.
\end{proof}

\section{Rank-One Computations}

In this section, we gather results from \cite{BB}, \cite{BBF} and \cite{CGco} and make some computations as a preparation to obtain functional equations of the Weyl group multiple Dirichlet series in the next section.

\medskip

For $j \in \mathbb Z_{>0}$, $\Psi \in \mathcal{M}_j(\Omega)$ and $a \in \mathfrak o_S$, we define
\[ \mathcal{D}(s,a ; \Psi , j) = \sum_{0 \neq \, c \, \in \frak o_S/ \frak o_S^\times} g(a,c;j) \Psi(c) |c|^{-s} |a|^{s/2}  .\] This series is absolutely convergent for $\frak{Re}(s) > 3/2$.
Let $m = n/\mathrm{gcd}(n,j)$ and set
\[ G_m(s) = \left (  ( 2 \pi )^{-(m -1)( s-1) } \Gamma (m s-m) /\Gamma (s-1)  \right )^{ [F: \mathbb Q]/2}.\]
Define  \[ \mathcal{D}^*(s, a ; \Psi, j)= G_m(s) \zeta_F(m s -m+1) \mathcal{D}(s,a,\Psi, j),\]
where $\zeta_F$ is the Dedekind zeta function of $F$.
If $\Psi \in  \mathcal{M}_j(\Omega)$ and $\eta \in F^\times_S$, we define \[  \tilde{\Psi}_\eta(c) = (\eta , c)^j_S \Psi ( \eta^{-1} c^{-1} ).\] The following result is fundamental.

\begin{Thm} \cite{BB} \label{thm-pp}
The function $\mathcal{D}^*(s,a;\Psi, j)$ has a meromorphic continuation to $\mathbb C$ and is holomorphic except for possible simple poles at $s= 1 \pm 1/m$. Moreover, there exist $S$-Dirichlet polynomials
$P(s; a \eta, j) $  such that \[
\mathcal{D}^*(s,a; \Psi, j)=\sum_{ \eta \in F_S^\times / F_S^{\times, n} }P(s; a \eta, j) \mathcal{D}^*(2-s,a;\tilde{\Psi}_\eta,j).\]
\end{Thm}

Let $\mathbf m =(m_1, \cdots , m_r) \in (\mathfrak o_S)^r \cap (F_S^\times)^r$ be fixed for the rest of this section. Let $\mathfrak A$ be the ring of Laurent polynomials in $|\varpi_v|^{s_i}$, $i=1, ..., r$, where $v$ runs over the places in $S_{\mathrm{fin}}$.
We define \begin{equation} \label{eqn-function-space} \frak{M}_B(\Omega)=\mathfrak{A} \otimes \mathcal{M}_B(\Omega) .\end{equation} Then it is natural to set $\frak{M}_t (\Omega)=\mathfrak{A} \otimes \mathcal{M}_t(\Omega)$ for $t \in \mathbb Z_{>0}$. We write \[ s_i= \alpha_i(\bf s)\qquad \text{ for } \mathbf s \in \frak h \]
and regard an element of $\frak{M}_B(\Omega)$ as a function on $\mathfrak h \times (F_S^\times)^r$.
Denote by $\iota$ the diagonal embedding: \[ \iota : F_S^\times \rightarrow (F_S^\times)^r, \qquad x \mapsto (x, x, \cdots , x) .\]
If $\Psi \in \frak M_B(\Omega)$ and $\mathbf a = (a_1, \cdots , a_r) \in ( \frak o_S /\frak o_S^\times)^r$, we define
\[ \Psi_i^{\mathbf a}(\mathbf s; c) = (\mathbf a, \iota(c))_{S,i}^B \ \Psi( \mathbf s; a_1, \cdots, a_i c , \cdots , a_r) ,\] where the notation $(\cdot, \cdot)^B_{S,i}$ is defined in (\ref{eqn-prev}).

\begin{Lem} \cite{BBF}
We have $\Psi_i^{\mathbf a} \in \frak M_{b_i}(\Omega)$.
\end{Lem}

\medskip

We define a shifted action of $W$ on $\frak h$ by \[ \sigma_i \circ \mathbf s = \sigma_i(\mathbf s - \rho^\vee) + \rho^\vee. \] Now we define an action of $\sigma_i$ on $\frak{M}_B(\Omega)$ as follows. For $\Psi \in \frak{M}_B(\Omega)$, we set
\[
(\sigma_i \Psi) (\mathbf s; \mathbf a) = \sum_{\eta \in
F_S^\times / F_S^{\times, n} } (\iota(\eta) , \mathbf a)_{S,i}^B \ P( s_i \, ; \, \eta m_i \mathbf a^{-h_i}, b_i) \ \Psi ( \sigma_i \circ \mathbf s; a_1, \cdots,  a_i \eta^{-1}, \cdots , a_r)  \]
where \[\mathbf a^{-h_i} = \prod_j a_j^{-\alpha_j(h_i)} =\prod_j a_j^{-a_{ij}}.\]
The element $\mathbf a^{-h_i}$ was denoted by $b_i$ in \cite{CGco}, while we set $b_i= (\alpha_i, \alpha_i)$ in this paper.

The following proposition plays a crucial role in describing functional equations.
\begin{Prop} \cite{BBF}
If $\Psi \in \frak M_B(\Omega)$ then $\sigma_i \Psi \in \frak M_B(\Omega)$.
\end{Prop}

\medskip

If $\mathbf c=(c_1, \cdots , c_r) \in (\frak o_S)^r$ and $\mathbf s \in \mathfrak h$, we set \begin{equation} \label{eqn-abs} |\mathbf c|_Q^{\mathbf s} = \prod_{i=1}^r |c_i|^{\, \alpha_i(\mathbf s)} =|c_1|^{s_1} \cdots |c_r|^{s_r} \quad \text{ and } \quad |\mathbf c|_P^{\mathbf s} = \prod_{i=1}^r |c_i|^{\, \omega_i(\mathbf s)} .\end{equation}
We let $\hat{\mathbf c}$ be the $(r-1)$-tuple $(c_1, \cdots, \hat{c}_i , \cdots , c_r)$ for $\mathbf c = (c_1, \cdots , c_r)$, where the hat on $c_i$ indicates that this entry is omitted. Let $\Psi \in \frak M_B(\Omega)$. 
We define \begin{eqnarray*} \mathcal{E}(s_i , \hat{\mathbf c}; \mathbf m, \Psi , i) &=& \sum_{0 \neq c_i \in \frak o_S /\frak o_S^\times} H(c_1, \cdots, c_i, \cdots , c_r ; \mathbf m) \Psi (\mathbf s ;c_1, \cdots , c_i , \cdots , c_r) |c_i|^{-s_i} \\ &=& \sum_{0 \neq c_i \in \frak o_S /\frak o_S^\times} H(\mathbf c ; \mathbf m) \Psi (\mathbf s ;\mathbf c)   |c_i|^{-s_i} .\end{eqnarray*} 
We also define \begin{eqnarray*} \tilde {\mathcal{E}}(\mathbf s , \hat{\mathbf c}; \mathbf m, \Psi , i) &=& \sum_{0 \neq c_i \in \frak o_S /\frak o_S^\times} H(c_1, \cdots, c_i, \cdots , c_r ; \mathbf m) \Psi (\mathbf s ;c_1, \cdots , c_i , \cdots , c_r) |c_i|^{-s_i}|\hat{\mathbf c} |_Q^{- \mathbf s } |\mathbf m|_P^{\mathbf s} \\ &=& \sum_{0 \neq c_i \in \frak o_S /\frak o_S^\times} H(\mathbf c ; \mathbf m) \Psi (\mathbf s ;\mathbf c) |\mathbf c|_Q^{- \mathbf s } |\mathbf m|_P^{\mathbf s} .\end{eqnarray*}Let $m= m(\alpha_i)$ and $s_i =\alpha_i(\mathbf s)$. We set
\[\mathcal{E}^*(s_i, \hat{\mathbf c}; \mathbf m , \Psi , i ) =  G_m(s_i) \zeta_F( m s_i - m+1) \mathcal{E}(s_i , \hat{\mathbf c}; \mathbf m , \Psi , i), \text{ and } \]
\[ \tilde{\mathcal{E}}^*(\mathbf s, \hat{\mathbf c}; \mathbf m , \Psi , i ) =  G_m(s_i) \zeta_F( m s_i - m+1) \tilde{\mathcal{E}}(\mathbf s , \hat{\mathbf c}; \mathbf m , \Psi , i). \]

The following proposition adapts the functional equation (Theorem \ref{thm-pp}) of the Kubota's Dirichlet series $\mathcal{D}^*(s,a;\Psi, j)$ for our purpose.

\begin{Prop}  \label{prop-ultimate}
Let $C=\prod_{j \neq i} c_j^{-a_{ij}}$. Then the function $\mathcal{E}^*(s_i,  \hat{\mathbf c}; \mathbf m , \Psi , i )$ (resp. $ \tilde{\mathcal{E}}^*(\mathbf s,  \hat{\mathbf c}; \mathbf m , \Psi , i )$) has a meromorphic continuation to $\mathbb C$ (resp. $\mathfrak h$) and is holomorphic except for possible simple poles at $s_i= 1 \pm 1/m$. Moreover, 
we have the functional equations
\begin{equation} \label{eqn-ff}  \mathcal{E}^*(s_i,  \hat{\mathbf c}; \mathbf m , \Psi , i ) =  |C m_i|^{1-s_i} \mathcal{E}^*(2-s_i, \hat{\mathbf c}; \mathbf m , \sigma_i \Psi , i ) \text{ and } \end{equation}
\begin{equation} \label{eqn-fs} \tilde{\mathcal{E}}^*(\mathbf s,  \hat{\mathbf c}; \mathbf m , \Psi , i ) =  \tilde{\mathcal{E}}^*(\sigma_i \circ \mathbf s, \hat{\mathbf c}; \mathbf m , \sigma_i \Psi , i ) . \end{equation}
\end{Prop}

\begin{proof}
With Theorem \ref{thm-3-last} established, the meromorphic continuation and functional equation of $\mathcal{E}^*$ follows from Theorem 5.8 of \cite{CGco}. Now the meromorphic continuation of $\tilde{\mathcal E}^*$ is clear, and its functional equation is obtained from that of $\mathcal E^*$ by direct computation.
\end{proof}

\section{The Multiple Dirichlet Series}

In this section, we define Weyl group multiple Dirichlet series from the local coefficients via twisted multiplicativity, and prove functional equations and meromorphic continuation of the multiple Dirichlet series.  As mentioned in the introduction, it is expected that our multiple Dirichlet series would be related to a Whittaker function up to a normalizing factor in the affine case. However, it is also expected that, in the indefinite case, the contribution coming from the imaginary roots is much more complicated. 

\medskip

Let $\mathbf m =(m_1, \cdots , m_r) \in (\mathfrak o_S)^r$ be fixed. Then we have $\mathbf m = \mathbf u \prod_\varpi \varpi_P^{\lambda_\varpi}$, $\mathbf u \in (\mathfrak o_S^\times)^r$,  with $\lambda_\varpi \in P_+$.
Let $\Psi \in \mathfrak{M}_B(\Omega)$, and we define a function $Z(\mathbf s; \mathbf m, \Psi)$ on $\mathfrak h$ by \[ Z(\mathbf s; \mathbf m, \Psi)= \sum_{\mathbf c} H(\mathbf c; \mathbf m) \Psi(\mathbf s; \mathbf c) |\mathbf c|_Q^{- \mathbf s} | \mathbf m|_P^{\mathbf s} ,\] where the sum is over $\mathbf c=(c_1, \cdots , c_r)$ such that $0 \neq c_i \in \frak o_S / \frak o_S^\times$ for $i=1, ... , r$. Although this is a function on $\mathfrak h$, except the factor $|\mathbf m|_P^\mathbf{s}$, it only depends on $(s_1,\cdots, s_r)$.

In order to investigate the convergence of $Z(\mathbf s; \mathbf m, \Psi)$, we first prove the following lemma.

\begin{Lem} \label{lem-conv}
Let $q$ be the norm of a prime $\varpi \in \frak o_S$, and consider
\[ EV_q \left ( \frac { D(\mathbf x)}{\Delta(\mathbf x)} , \mathbf s \right ) = \prod_{\alpha \in \Phi_+} \left ( \frac {1 - q^{m(\alpha) \alpha ( \rho^\vee - \mathbf s) -1}  }{1 - q^{m(\alpha) \alpha ( \rho^\vee - \mathbf s)}  } \right )^{\mathrm{mult}(\alpha)}.  \] Assume that $\frak{Re}(\alpha_i(\mathbf s)) > 1$ for each $i=1, \dots , r$.  Then the product $EV_q \left ( \frac { D(\mathbf x)}{\Delta(\mathbf x)} , \mathbf s \right )$ is absolutely convergent for sufficiently large $q$. In this case, if $\frak{Re}(\alpha_i(\mathbf{s}))>1+\epsilon$ for each $1\leq i\leq r$, we have \[ |EV|_q \left ( \frac { D(\mathbf x)}{\Delta(\mathbf x)} , \mathbf s \right ) \sim 1 + O(q^{-\varepsilon}), \] where the implicit constant only depends on $\epsilon$.
\end{Lem}

\begin{proof}
We may assume that $\mathbf s$ is real and that $\frak{Re}(\alpha_i(\mathbf s)) \ge 1 + \log_q r + \varepsilon$ for sufficiently large $q$ and for each $i=1, \dots , r$. Since we have
\[  \frac {1 - q^{m(\alpha) \alpha ( \rho^\vee - \mathbf s) -1}  }{1 - q^{m(\alpha) \alpha ( \rho^\vee - \mathbf s)}  } = 1 +  \frac {q^{m(\alpha) \alpha ( \rho^\vee - \mathbf s) } (1-q^{-1})  }{1 - q^{m(\alpha) \alpha ( \rho^\vee - \mathbf s)}  } .\]
Therefore
\[ \log\left(EV_q \left ( \frac { D(\mathbf x)}{\Delta(\mathbf x)} , \mathbf s \right ) \right)\sim \sum_{\alpha \in \Phi_+} \mathrm{mult}(\alpha) \ \frac {q^{m(\alpha) \alpha ( \rho^\vee - \mathbf s) } }{1 - q^{m(\alpha) \alpha ( \rho^\vee - \mathbf s)}  }. \] 
For each $k \in \mathbb Z_{>0}$, there are at most $r^k$ linearly independent root vectors all together for positive roots with depth $k$ in the Kac-Moody algebra $\mathfrak g(A)$, that is to say,
\[ \sum_{\alpha \in \Phi_+ \atop d(\alpha) =k} \mathrm{mult}(\alpha) \le r^k .\]
Then we have \begin{eqnarray*}
\sum_{\alpha \in \Phi_+} \mathrm{mult}(\alpha) \ \frac {q^{m(\alpha) \alpha ( \rho^\vee - \mathbf s) } }{1 - q^{m(\alpha) \alpha ( \rho^\vee - \mathbf s)}  }  &\le &\sum_{\alpha \in \Phi_+} \mathrm{mult}(\alpha) \ q^{\alpha(\rho^\vee -\mathbf s)} \\ & \le & \sum_{k=1}^\infty r^k q^{-k( \log_q r+\varepsilon)} = \sum_{k=1}^\infty q^{-k\varepsilon} = \frac 1 {q^{\varepsilon}-1}\leq Aq^{-\epsilon},
\end{eqnarray*}
where $A$ could be chosen to be $(1-2^{-\epsilon})^{-1}$. The last assertion of the proposition follows from this by taking the natural exponential function and we see that the implied constant only depends on $\epsilon$, that is, independent of $q$.
\end{proof}

\begin{Thm} \label{thm-conv-final}
Assume that $\Psi \in \mathfrak{M}_B(\Omega)$. The series $Z(\mathbf s; \mathbf m, \Psi)$ absolutely converges for $\mathbf s \in \mathfrak h$ satisfying the condition:
\[ \frak{Re}(\alpha_i(\mathbf s))=\frak{Re}(s_i) > \max\{1+\log_2r, 2\} \qquad \text{ for each }i=1, ... , r.\]
\end{Thm}

\begin{proof}
We may assume that $\mathbf s$ is real.
Since the function $\Psi$ is bounded, it is sufficient to consider
\[ \sum_{\mathbb c} |H(\mathbf c; \mathbf m)| |\mathbf c|_Q^{-\mathbf s} = \prod_\varpi \sum_{\beta \in Q_+} | H(\varpi_Q^\beta; \varpi_P^{\lambda_\varpi})| |\varpi|^{-\beta (\mathbf s)} ,\]
which only depends on $(s_1,\cdots, s_r)$. We fix $\varpi$ for the time being, and write $q=|\varpi|$.

Suppose that $w=\sigma_{i_1} \cdots \sigma_{i_k} \in W$ is a reduced expression. Then one can show that
\[ w^{-1} \cdot \beta = \mu_{i_1}(\beta) \sigma_{i_k} \cdots \sigma_{i_2}\alpha_{i_1}  +  \cdots + \mu_{i_{k-1}}(\beta) \sigma_{i_k} \alpha_{i_{k-1}} + \mu_{i_k}(\beta) \alpha_{i_k} +\beta. \]
Using this identity with $\beta=0$, let us prove that, for any $M>0$, if $\ell(w)$ is large, then
\begin{equation} \label{eqn-qp}
(w^{-1} \cdot 0 ) ( \rho^\vee - \mathbf s) \le -M\ell(w)\cdot \min\{s_i-1 \colon 1\leq i\leq r\}.
\end{equation} Note first that it holds trivially for any $w$ if $M=1$. In general, since $w^{-1}\cdot 0$ is the weighted sum of all roots in $\Phi(w)$ by $\mu_i(0) \geq 1$, $i=1, \dots , r$, it is enough to show that as $ \ell(w)\rightarrow \infty$,
\[\frac{1}{\ell(w)}\sum_{\alpha\in \Phi(w)}d(\alpha)\rightarrow \infty.\]
Indeed, for any $M>0$, let $M'=\{\alpha\in \Phi_+\colon d(\alpha)\leq 2M\}$. If $\ell(w)>2M'$, then
\[\sum_{\alpha\in \Phi(w)}d(\alpha)>2M(\ell(w)-M')>M\ell(w),\]
and (\ref{eqn-qp}) is proved.

We first assume that $q$ is sufficiently large so that we have
\[  s_i \ge 2+\log_q r + \varepsilon \qquad \text{ for each }i=1, ... , r.\]
From the definition of $H(\varpi_Q^\beta; \varpi_P^{\lambda_\varpi})$, we have
\begin{eqnarray*} \sum_{\beta \in Q_+} H(\varpi_Q^\beta; \varpi_P^{\lambda_\varpi}) q^{-\beta(\mathbf s)} &=& EV_q \left ( N(\mathbf x; \lambda), \mathbf s \right )  \\ & = & EV_q \left ( \frac {D(\mathbf x)}{\Delta(\mathbf x)} , \mathbf s \right ) EV_q \left ( \sum_{w \in W} j(w, \mathbf x) (1|_\lambda w)(\mathbf x), \mathbf s \right ) .\end{eqnarray*}

We obtain from  Lemma \ref{lem-conv},  Proposition \ref{prop-abc} and \eqref{eqn-qp} that
\begin{eqnarray}
\sum_{\beta \in Q^+} | H(\varpi_Q^\beta; \varpi_P^{\lambda_\varpi})| q^{-\beta (\mathbf s)} &\le & |EV|_q \left ( \frac {D(\mathbf x)}{\Delta(\mathbf x)} , \mathbf s \right ) |EV|_q \left ( \sum_{w \in W} j(w, \mathbf x) (1|_\lambda w)(\mathbf x), \mathbf s \right ) \nonumber \\
 &\le &  |EV|_q \left ( \frac {D(\mathbf x)}{\Delta(\mathbf x)} , \mathbf s \right ) \sum_{w \in W} 3^{\ell(w)} q^{(w^{-1} \cdot 0)(\rho^\vee - \mathbf s)} \nonumber \\ & \le &  |EV|_q \left ( \frac {D(\mathbf x)}{\Delta(\mathbf x)} , \mathbf s \right ) \sum_{w \in W} 3^{\ell(w)} q^{-\ell(w)(1+\varepsilon)} \nonumber \\ & \le & |EV|_q \left ( \frac {D(\mathbf x)}{\Delta(\mathbf x)} , \mathbf s \right ) \left ( 1+ \sum_{k=1}^\infty \left ( \frac { 3 r}{q^{1+\varepsilon}} \right )^k  \right ) \nonumber \\ &\sim & 1+O (q^{-1-\varepsilon}) \label{eqn-geul}
\end{eqnarray}
for sufficiently large $q$. Here we applied the trivial case $M=1$ of (\ref{eqn-qp}).

For a general $q$, since $\frak{Re}(\alpha_i(\mathbf s))>1+\log_2r\geq 1+\log_qr$ for each $i$, the same estimation as in Lemma \ref{lem-conv} shows that
\[ |EV|_q \left ( \frac {D(\mathbf x)}{\Delta(\mathbf x)} , \mathbf s \right ) \]
is absolutely convergent. Using (\ref{eqn-qp}) and choosing $M$ large enough so that $6r<q^{M(1+\epsilon)}$,
\[|EV|_q \left ( \sum_{w \in W} j(w, \mathbf x) (1|_\lambda w)(\mathbf x), \mathbf s \right )\ll \sum_{w \in W} 3^{\ell(w)} q^{-M\ell(w)(1+\varepsilon)}\ll  \left ( 1+ \sum_{k=1}^\infty \left ( \frac { 3 r}{q^{M(1+\varepsilon)}} \right )^k  \right )\ll 1. \]
Therefore, all factors in \[ \prod_\varpi \sum_{\beta \in Q^+} | H(\varpi_Q^\beta; \varpi_P^{\lambda_\varpi})| |\varpi|^{-\beta (\mathbf s)} \]  are absolutely convergent, and the absolute convergence of the whole product is obtained using (\ref{eqn-geul}) .
\end{proof}

For any $\alpha \in \Phi$, we define \[ \zeta_\alpha (\mathbf s) = \zeta_F \left ( 1 + m(\alpha) \alpha(\mathbf s - \rho^\vee ) \right ) \quad \text{and} \quad G_\alpha (\mathbf s) =G_{m(\alpha)}  \left (1  +   \alpha(\mathbf s - \rho^\vee ) \right ).\] 
It is easy to see that $G_\alpha ( \sigma_i \circ \mathbf s) = G_{\sigma_i \alpha}(\mathbf s)$ and $\zeta_\alpha( \sigma_i \circ \mathbf s) = \zeta_{\sigma_i \alpha} ( \mathbf s)$.
Then we have, for $w \in W$,
 \begin{equation} \label{eqn-hoho} G_\alpha ( w \circ \mathbf s) = G_{w^{-1} \alpha}(\mathbf s) \quad \text{ and } \quad \zeta_\alpha( w \circ \mathbf s) = \zeta_{w^{-1} \alpha} ( \mathbf s).\end{equation}
In particular, $G_{\alpha_i}(\sigma_i \circ \mathbf s) = G_{-\alpha_i}( \mathbf s)$ and $\zeta_{\alpha_i}( \sigma_i \circ \mathbf s) = \zeta_{- \alpha_i} ( \mathbf s)$.
Now we define
\[ G(w,\mathbf s) = \prod_{ \alpha \in \Phi(w)} \frac{ G_\alpha (\mathbf s) \zeta_{\alpha}(\mathbf s) } {G_{-\alpha} (\mathbf s) \zeta_{-\alpha}(\mathbf s) }. \]

Let $\frak{F} = \{ \mathbf s \in \frak h_{\mathbb R} \ | \   \alpha_i ( \mathbf s ) =s_i \ge 1 \text{ for all } i=1, ... , r \}$. 
We define the {\em shifted Tits cone} $\frak{X} \subseteq \frak h_{\mathbb R}$ to be  \[ \frak{X} =  \bigcup_{w \in W} w \circ \frak{F}. \]

\begin{Prop} \cite{Bourbaki, Kac}
The shifted Tits cone $\frak{X}$ is a convex cone, and we have $\frak{X} = \frak h_{\mathbb R}$ if and only if $|W| < \infty$.
\end{Prop}

Let $ \frak{L} = \left\{ \mathbf s \in \frak h_{\mathbb R} \ | \   \alpha_i ( \mathbf s )=s_i  > \max\{2, 1+\log_2r\} \text{ for all } i=1, ... , r \right\}  $, and we define \[ \frak{X}_0 = \mathrm{convex\  hull} \left ( \bigcup_{w \in W} w \circ \frak{L} \right ) \subseteq \frak{X} . \]

\begin{Thm} \label{thm-main-final}
The Dirichlet series $Z(\mathbf s; \mathbf m, \Psi)$ has meromorphic continuation to all $\mathbf s \in \frak h$ such that $\frak{Re}(\mathbf s) \in \frak{X}_0$ and satisfies the functional equation \[ Z(w \circ \mathbf s; \mathbf m, w \Psi) = G(w, \mathbf s)Z(\mathbf s; \mathbf m, \Psi) \] for each $w \in W$, where the action of $w$ on $\Psi$ is given by the composition of simple reflections $\sigma_i$. The set of polar hyperplanes is contained in the $W$-translates of the hyperplanes $s_i= 1 \pm m(\alpha_i)/n$.
\end{Thm}

\begin{proof} If $r=1$, we essentially have the Kubota's series and we are done in this case. Let us assume $r\geq 2$ from now on. It follows that $1+\log_2r\geq 2$.

The Dirichlet series defining $Z(\mathbf s; \mathbf m, \Psi)$ is absolutely convergent by Theorem \ref{thm-conv-final} in the region $\Lambda_0 = \{ \mathbf s \in \frak{h} \ | \ \frak{Re}(\mathbf s)  \in \frak{L} \}$.  We write
\begin{eqnarray}
Z(\mathbf s; \mathbf m, \Psi) &=& \sum_{\mathbf c} H(\mathbf c; \mathbf m) \Psi(\mathbf s; \mathbf c) |\mathbf c|_Q^{- \mathbf s} | \mathbf m|_P^{\mathbf s} \nonumber \\ &=& \sum_{\hat{\mathbf c} } \sum_{0 \neq c_i \in \frak o_{S}/\frak o_S^{\times} } H(\mathbf c; \mathbf m) \Psi(\mathbf s; \mathbf c) |\mathbf c|_Q^{- \mathbf s} | \mathbf m|_P^{\mathbf s} \nonumber \\ &=& |\mathbf m|_P^\mathbf s\sum_{\hat{\mathbf c} }  \frac{\mathcal{E}(s_i , \hat{\mathbf c}; \mathbf m, \Psi , i)}{\prod_{j\neq i}|c_j|^{s_j}}, \label{eqn-eee}
\end{eqnarray}
where $\hat{\mathbf c}$ is the $(r-1)$-tuple $(c_1, \cdots, \hat{c}_i , \cdots , c_r)$ with $c_i$ omitted. 

With the expression (\ref{eqn-eee}), let us first continue $Z(\mathbf s; \mathbf m,\Psi)$ to a larger region as follows. We consider the modified series 
\[Z_i(\mathbf s;\mathbf m,\Psi)=(s_i-1-m(\alpha_i)/n)(s_i-1+m(\alpha_i)/n)Z(\mathbf s;\mathbf m,\Psi), \text{ and }\]
\[\mathcal{F}(s_i,\hat{\mathbf c};\mathbf m, \Psi, i)=(s_i-1-m(\alpha_i)/n)(s_i-1+m(\alpha_i)/n)\mathcal{E}(s_i, \hat{\mathbf c};\mathbf m, \Psi, i).\]
By Proposition \ref{prop-ultimate}, $\mathcal{F}$ is analytic on $\mathbb C$ with a functional equation.
We claim that for any $\epsilon>0$, on the region with $\frak{Re}(s_i)\geq1+\log_2r+\epsilon$, we have
\[\mathcal{F}(s_i, \hat{\mathbf c};\mathbf m, \Psi, i)=O \left (\prod_{j\neq i}|c_j|^{1+\log_2r+\epsilon} \right ),\] where the implied constant is independent of $\hat{\mathbf c}$. Actually, since the multiple sum \[ \sum_{\mathbf c} H(\mathbf c; \mathbf m) \Psi(\mathbf s; \mathbf c) |\mathbf c|_Q^{- \mathbf s} | \mathbf m|_P^{\mathbf s} \] is absolutely convergent on the region with $\frak{Re}(s_j)\geq1+\log_2r+\epsilon$ ($1\leq j\leq r$), the general term 
\[  \frac{\mathcal{E}(s_i , \hat{\mathbf c}; \mathbf m, \Psi , i)}{\prod_{j\neq i}|c_j|^{s_j}}\] over $\hat{\mathbf c}$ in \eqref{eqn-eee} is then bounded and the claim follows by taking the infimum. 
By the functional equation, we have for $\frak{Re}(s_i)\leq 1-\log_2r-\epsilon$,
\[\mathcal{F}(s_i, \hat{\mathbf c};\mathbf m, \Psi, i)=O \left (|C|^{1-\frak{Re}(s_i)}\prod_{j\neq i}|c_j|^{1+\log_2r+\epsilon} \right )=O\left(\prod_{j\neq i}|c_j|^{-a_{ij}(1-\frak{Re}(s_i))+1+\log_2r+\epsilon}\right).\]
In particular, on the line $\frak{Re}(s_i)=1-\log_2r-\epsilon$ we have
\[\mathcal{F}(s_i, \hat{\mathbf c};\mathbf m, \Psi, i)=O\left(\prod_{j\neq i}|c_j|^{1+(1-a_{ij})(\log_2r+\epsilon)}\right).\] Now by Phragm\'{e}n-Lindel\"{o}f Theorem, on the strip defined by $1-\log_2r-\epsilon\leq \frak{Re}(s_i)\leq 1+\log_2r+\epsilon$, we have the bound
\[\mathcal{F}(s_i, \hat{\mathbf c};\mathbf m, \Psi, i)=O\left(\prod_{j\neq i}|c_j|^{1+(1-a_{ij})(\log_2r+\epsilon)}\right).\] Finally, by putting the above bounds into $Z_i$ and considering the expression (\ref{eqn-eee}), we then see that for any $u\in\mathbb R$, there exists a positive real number $u'$ such that the series $Z_i$ (over $\hat{\mathbf c}$) is absolutely convergent on the region
\[A(u,u')=\{\mathbf s\in\mathfrak h: \frak{Re}(\alpha_i(\mathbf s))>u, \ \frak{Re}(\alpha_j(\mathbf s))>u', j\neq i \}.\] Note here that the dependence on variables other than $s_i$'s in $|\mathbf m|_P^\mathbf{s}$ does not affect the absolute convergence. This provides the analytic continuation of $Z_i$ and hence the meromorphic continuation of $Z$, from $\Lambda_0$ to $\Lambda_0\cup A(u,u')$.

By the functional equations in Proposition \ref{prop-ultimate} again, $Z$ can be meromorphically continued to $\sigma_i\circ \Lambda_0$. Now we choose $u$ so that  $A(u,u')$ intersects both $\Lambda_0$ and $\sigma_i\circ\Lambda_0$. Since the union of these pieces is (simply) connected, the function $Z_i$ is analytic on the union of these regions. Applying Bochner's  Tube-domain Theorem, we see that we may continue $Z$ meromorphically to $\Lambda_i$, the convex hull of $\Lambda_0\cup\sigma_i \circ\Lambda_0$. On this region, we have, by Proposition \ref{prop-ultimate},
\begin{eqnarray*} G(\sigma_i, \mathbf s) Z(\mathbf s; \mathbf m, \Psi) &=& \frac 1 { G_{-\alpha_i}( \mathbf s) \zeta_{- \alpha_i}(\mathbf s) } \sum_{\hat{\mathbf c} }  \mathcal{\tilde{E}}^*(\mathbf s , \hat{\mathbf c}; \mathbf m, \Psi , i)  \\ &=& \frac 1 { G_{\alpha_i}( \sigma_i \circ \mathbf s) \zeta_{\alpha_i}(\sigma_i \circ \mathbf s)} \sum_{\hat{\mathbf c} }  \mathcal{\tilde{E}}^*(\sigma_i \circ \mathbf s , \hat{\mathbf c}; \mathbf m, \sigma_i \Psi , i) \\ &=& \sum_{\hat{\mathbf c} }  \mathcal{\tilde{E}}(\sigma_i \circ \mathbf s , \hat{\mathbf c}; \mathbf m, \sigma_i \Psi , i) \\ &=& Z( \sigma_i \circ \mathbf s, ; \mathbf{m}, \sigma_i \Psi ).
\end{eqnarray*}

Now we extend the meromorphic continuation to $\frak{Re}(\mathbf s) \in \frak{X}_0$. The argument is similar to that of the finite case (\cite{BBCFH,BBF}), though we need to make some modifications. If $\Lambda_1=\mathfrak h$ or $\Lambda_2=\mathfrak h$, we are done. If not, we consider $\Lambda_1\cup\Lambda_2$. It is easy to see that the union is simply connected since both are convex. Because they intersect nontrivially, $Z$ is continued to the union. We multiply $Z$ by linear factors to eliminate the possible poles on such a union, and then apply Bochner's Theorem. Then consider $\Lambda_3$ and the union $\Lambda_1\cup\Lambda_2\cup\Lambda_3$ similarly. At the end, either we stop somewhere and we have continuation to $\mathfrak h$ or we have continuation to the convex hull, say $\Lambda_0^{(1)}$, of $\cup_i\Lambda_i$. In the latter case, repeat the whole procedure from the beginning by replacing $\Lambda_0$ by $\Lambda_0^{(1)}$, except that all the translates of $\Lambda_0^{(1)}$ intersect $\Lambda_0^{(1)}$ non-trivially and we do not need to do the growth estimate of the $\mathcal{E}$-series. 

Such a process may not stop in finitely many steps and it is clear that we need more and more linear factors to eliminate the poles of $Z$ as we carry out this procedure. However, since any $\mathbf s \in \frak h$ with $\frak{Re}(\mathbf s) \in \frak{X}_0$ is contained in a union of a finite number of $W$-translates of $\Lambda_i$'s, we need at most finitely many steps to continue $Z$ to $\mathbf s$, therefore at most finitely many linear factors in such a continuation.

The functional equation
$Z(w \circ \mathbf s; \mathbf m, w \Psi) = G(w, \mathbf s)Z(\mathbf s; \mathbf m, \Psi)$ for each $w \in W$ can be proved using (\ref{eqn-phi}), (\ref{eqn-hoho}) and induction, along with the continuation. The possible polar hyperplanes are clear from our continuation and the poles of Kubota's series. We are done with the proof.
 \end{proof}

\begin{Exa}
Let us consider the hyperbolic Kac-Moody root system associated with the generalized Cartan matrix $A= \begin{pmatrix} 2 & -3 \\ -3 & 2 \end{pmatrix}$. Then the Weyl group $W$ is the free group generated by the simple reflections $\sigma_1$ and $\sigma_2$. As in the proof of Theorem \ref{thm-main-final}, we let $\Lambda_i$  be the convex hull of $\Lambda_0 \cup ( \sigma_i \circ \Lambda_0)$ for $i=1, 2$. Then the projections on the real plane of the regions $\Lambda_1$ and $\Lambda_2$ are given by
\[ \Lambda_1 : \qquad 3 x + y > 5, \quad 3 x+2 y > 10, \quad y > 2 \] and \[ \Lambda_2 : \qquad x +3  y > 5, \quad 2 x+3 y > 10, \quad x > 2, \] where $x= \frak {Re}(s_1)$ and $y=\frak{Re}(s_2)$. The shifted Tits cone $\frak X$ is given by
\[ \frak X : \qquad  (3+\sqrt 5) x + 2 y > 5+\sqrt 5 , \quad  (3-\sqrt 5) x + 2 y > 5-\sqrt 5.\] On the other hand, the boundary of the convex hull $\frak X_0$ is the piecewise linear curve connecting the points in each column of the following: \[ \begin{array} {lcl} (2,2)&  \phantom{LLLL}&(2,2) \\ (0,5) = \sigma_1 \circ (2,2)  &  & (5,0) = \sigma_2 \circ (2,2)  \\ (-3, 12) = \sigma_1 \sigma_2 \circ (2,2) & &(12, -3) = \sigma_2 \sigma_1 \circ (2,2) \\ (-10, 30) =  \sigma_1 \sigma_2 \sigma_1 \circ (2,2) &  &(30, -10) =  \sigma_2 \sigma_1 \sigma_2 \circ (2,2) \\ \hskip 1 cm  \vdots & & \hskip 1 cm \vdots
\end{array} \]
In particular, $(3/2, 3/2) \in \frak X \setminus \frak X_0$.

\end{Exa}

\begin{Rmk}
One method to investigate moment problem of Dirichlet $L$-functions $L(s,\chi_d)$ with quadratic characters over $\mathbb Q$ in $d$-aspect is to consider the multiple Dirichlet series of the form
\[\sum_d\frac{L(s_1,\chi_d)L(s_2,\chi_d)\cdots L(s_r,\chi_d)}{|d|^{s_{r+1}}}.\]
Sufficient meromorphic continuation of this series will produce asymptotic formulas for the $r$-th moment. However, only in the cases $r=1,2,3$, necessary continuation has been obtained. Actually, in those cases the series can be continued to the whole $\mathbb C^r$. If $r>3$, the group of functional equations becomes infinite, and we are not able to continue the series sufficiently at the present. For more details, see Bump's survey paper \cite{Bump}.

Remarkably, in a recent paper \cite{BD}, Bucur and Diaconu considered quadratic Dirichlet $L$-functions over rational function fields in the case $r=4$, and managed to continue the corresponding multiple Dirichlet series sufficiently.  This is the first such result where the group of functional equations is infinite.

More specifically, they consider two multiple Dirichlet series, one from the moment problem and the other from the corresponding Weyl group action. Since the latter can be sufficiently continued, by proving a uniqueness theorem (up to a single variable power series), they prove sufficient continuation of the multiple Dirichlet series from the moment problem.

Their Weyl group multiple Dirichlet series are constructed in a similar way with those considered in this paper. In their case, the root system is of affine type $D_4^{(1)}$, and by identifying $s_i$ with $\alpha_i$,
they put $\textbf s =(s_1, \dots , s_r)$ and
\begin{equation} \label{DDeF} \Delta(\textbf s)=\prod_{\alpha\in\Phi^{\mathrm{re}}_+}(1-q^{d(\alpha)-\alpha}) \quad \text{ and } \quad D(\textbf s)=\prod_{\alpha\in\Phi^{\mathrm{re}}_+}(1-q^{d(\alpha)-\alpha+1}). \end{equation}  Here $q$ is the order of the finite field, and we have changed their notations to be consistent with ours. 
One can compare \eqref{DDeF} with \eqref{DDe}.
Moreover, they define a matrix-valued  rational function $M(w,\textbf s)$ that satisfies the same cocycle relation as $j(w, \mathbf x)$ in this paper.
Then they construct the matrix series
\[Z(\mathbf s)=\sum_{w\in W}\frac{D(w\mathbf s)}{\Delta(w\mathbf s)}M(w,\mathbf s),\]
and proved its analytic continuation and functional equations, namely
\[Z(\mathbf s)=M(w,\mathbf s)Z(w\mathbf s), \text{\  \  for any } w\in W.\]

In general, over number fields, the multiple Dirichlet series obtained from the moment problem  are expected to be different from their counterparts constructed from the Weyl group actions. In the affine case $D_4^{(1)}$, they are expected to differ by a normalizing factor. In the indefinite case, the situation may be much more complicated.

\end{Rmk}

\end{document}